\documentclass[reqno,a4paper,final]{amsart}
\usepackage{amsthm}
\usepackage{mathtools}
\mathtoolsset{showonlyrefs,showmanualtags}
\usepackage[stixtwo]{fontsetup}
\usepackage{enumitem}
\usepackage[centering]{geometry} 
\usepackage{here}
\usepackage{microtype}
\usepackage{tikz}
\usetikzlibrary{calc}
\usetikzlibrary{cd}
\usepackage{url}
\usepackage[breaklinks=true,colorlinks=true,allcolors=black]{hyperref}
\usepackage{showlabels}
   
   \showlabels{bib}
\usepackage{amsrefs}
\title{Iwasawa theory for vertex-weighted graphs}
\date{}
\author{Ryosuke Murooka}
\address[Ryosuke Murooka]{Graduate School of Mathematics, Nagoya University, Furo-cho, Chikusaku, Nagoya, 464-8601, Japan}
\email{ryosuke.murooka.c1@math.nagoya-u.ac.jp}
\author{Sohei Tateno}
\address[Sohei Tateno]{Graduate School of Natural Science \& Technology, Kanazawa University, Kakuma-machi, Kanazawa, 920-1192, Japan}
\email{inu.kaimashita@gmail.com}
\keywords{graph theory, Iwasawa theory, vertex-weighted
graphs, Iwasawa-type formula, Kida’s formula}
\subjclass[2020]{Primary 5C22; Secondary 11R23, 5C25}

\newcommand\numberthis{\addtocounter{equation}{1}\tag{\theequation}} 

\allowdisplaybreaks 

 \numberwithin{equation}{section}
 \counterwithin{figure}{section}
 
\theoremstyle{plain}
\newtheorem{theorem}{Theorem}[section]
\newtheorem{proposition}[theorem]{Proposition}
\newtheorem{lemma}[theorem]{Lemma}
\newtheorem{corollary}[theorem]{Corollary}
\theoremstyle{definition}
\newtheorem{definition}[theorem]{Definition}
\newtheorem{definitions}[theorem]{Definitions}
\newtheorem{example}[theorem]{Example}
\theoremstyle{remark}
\newtheorem{remark}[theorem]{Remark}

\newcommand{\rank}{\operatorname{rank}\nolimits}
\newcommand{\inc}{\operatorname{inc}\nolimits}
\newcommand{\adj}{\operatorname{adj}\nolimits}
\newcommand{\val}{\operatorname{val}\nolimits}
\newcommand{\Ind}{\operatorname{Ind}\nolimits}
\newcommand{\GL}{\operatorname{GL}\nolimits}
\newcommand{\Gal}{\operatorname{Gal}\nolimits}
\newcommand{\Res}{\operatorname{Res}\nolimits}
\newcommand{\ord}{\operatorname{ord}\nolimits}
\newcommand{\Aut}{\operatorname{Aut}\nolimits}
\newcommand{\triv}{\operatorname{triv}\nolimits}
\newcommand{\reg}{\operatorname{reg}\nolimits}

\usepackage[toc=false, symbols,nogroupskip,sort=use]{glossaries-extra}
\usepackage{glossary-longbooktabs}
\makenoidxglossaries
\newglossarystyle{formel_altlong4colheader}{%
 \setglossarystyle{altlong4colheader}%
}
\glsxtrRevertMarks
\glsxtrnewsymbol[description={the Laplacian matrix of a vertex-weighted graph $X$}]{bfL}{\ensuremath{\mathbfit{L}_X}}
\glsxtrnewsymbol[description={the symmetrized Laplacian matrix of a vertex-weighted graph $X$}]{calL}{\ensuremath{\mathcal{L}_X}}
\glsxtrnewsymbol[description={the diagonal matrix associated to a vertex-weighted graph $X$}]{scrW}{\ensuremath{\mathscr{W}_X}}
\glsxtrnewsymbol[description={the degree matrix of a vertex-weighted graph $X$}]{bfD}{\ensuremath{\mathbfit{D}_X}}
\glsxtrnewsymbol[description={the weighted adjacency matrix of a vertex-weighted graph $X$}]{bfW}{\ensuremath{\mathbfit{W}_X}}
\glsxtrnewsymbol[description={the symmetrized weighted adjacency matrix of a vertex-weighted graph $X$}]{calW}{\ensuremath{\mathcal{W}_X}}
\glsxtrnewsymbol[description={the weighted boundary matrix of a vertex-weighted graph $X$}]{bfB}{\ensuremath{\mathbfit{B}_X}}
\glsxtrnewsymbol[description={the distance in a tree $T$ between vertices $u$ and $v$}]{dT}{\ensuremath{d_T(u,v)}}
\glsxtrnewsymbol[description={the rooted directed tree $T_v$ by orienting every edge of a tree $T$ towards the root $v$}]{Tv}{\ensuremath{T_v}}
\glsxtrnewsymbol[description={the weight of a rooted directed tree $T_v$}]{omegaTv}{\ensuremath{w(T_v)}}
\glsxtrnewsymbol[description={the rooted weight complexity of a vertex-weighted graph $X$}]{kappavX}{\ensuremath{\kappa_v(X)}}
\glsxtrnewsymbol[description={the weight complexity of a vertex-weighted graph $X$}]{kappaX}{\ensuremath{\kappa(X)}}
\glsxtrnewsymbol[description={the weight of a vertex $v$}]{wv}{\ensuremath{w_v}}
\glsxtrnewsymbol[description={the symmetric subgraph derived from a subset A}]{tildeA}{\ensuremath{\widetilde{A}}}
\glsxtrnewsymbol[description={the set of vertices of a directed graph $X$}]{VX}{\ensuremath{V(X)}}
\glsxtrnewsymbol[description={the set of directed of a directed graph $X$}]{bbEX}{\ensuremath{\mathbb{E}(X)}}
\glsxtrnewsymbol[description={the incidence map of a directed graph $X$}]{incX}{\ensuremath{\inc_X}}
\glsxtrnewsymbol[description={the origin of a directed edge $e$}]{oe}{\ensuremath{o(e)}}
\glsxtrnewsymbol[description={the terminus of a directed edge $e$}]{te}{\ensuremath{t(e)}}
\glsxtrnewsymbol[description={the inverse directed edge of a directed edge $e$}]{bare}{\ensuremath{\overbar{e}}}
\glsxtrnewsymbol[description={the undirected graph derived from a directed graph}]{Xundir}{\ensuremath{X_{\mathrm{undir}}}}
\glsxtrnewsymbol[description={the distance between a vertex $v$ from a vertex $u$ in a tree $T$}]{dTuv}{\ensuremath{d_T(u,v)}}
\glsxtrnewsymbol[description={the derived graph of a voltage assignment $\alpha$}]{Xalpha}{\ensuremath{X(\alpha)}}
\glsxtrnewsymbol[description={the summand of $\mathcal{W}_X$ corresponding to a voltage assignment $\alpha\colon \mathbb{E}(X) \to G$ and to $\sigma \in G$}]{calWXalphasigma}{\ensuremath{\mathcal{W}_X^{\alpha, \sigma}}}
\glsxtrnewsymbol[description={a multiplicative group isomorphic to $\mathbb{Z}_p^d$}]{Gamma}{\ensuremath{\Gamma}}
\glsxtrnewsymbol[description={the quotient $\Gamma/\Gamma^{p^n}$}]{Gamman}{\ensuremath{\Gamma_n}}
\glsxtrnewsymbol[description={the formal power series corresponding to $\sigma \in\Gamma$}]{bbtbma}{\ensuremath{\mathbb{t}(\sigma)(T_1,\dotsc, T_d)}}
\glsxtrnewsymbol[description={the weighted matrix associated to a voltage assignment $\alpha$ with formal power series entries}]{bbWXalpha}{\ensuremath{\mathbb{W}_X^{\alpha}(T_1,\dotsc, T_d)}}
\glsxtrnewsymbol[description={the formal power series associated to a voltage assignment $\alpha$ with formal series entries}]{QXalpha}{\ensuremath{Q_X^{\alpha}(T_1,\dotsc, T_d)}}
\glsxtrnewsymbol[description={the element in $(\overbar{\mathbb{Q}}_p^*)^d$ associated to a character $\psi \in \widehat{\Gamma}_n$}]{bmzetapsi}{\ensuremath{\mathbfit{\zeta}_{\psi}}}
\glsxtrnewsymbol[description={the rooted Iwasawa $\mu$-invariant of $(X,\alpha)$ at $v$}]{muvXalpha}{\ensuremath{\mu_{v}(X,\alpha)}}
\glsxtrnewsymbol[description={the rooted Iwasawa $\lambda$-invariant of $(X,\alpha)$ at $v$}]{lambdavXalpha}{\ensuremath{\lambda_{v}(X,\alpha)}}
\glsxtrnewsymbol[description={the Iwasawa $\mu$-invariant of $(X,\alpha)$}]{muXalpha}{\ensuremath{\mu(X,\alpha)}}
\glsxtrnewsymbol[description={the Iwasawa $\lambda$-invariant of $(X,\alpha)$}]{lambdaXalpha}{\ensuremath{\lambda(X,\alpha)}}
\glsxtrnewsymbol[description={the polynomial in $t$ associated to a vertex-weighted graph $X$, to a voltage assignment $\alpha\colon \mathbb{E}(X) \to G$, and to a representation $\rho$ of G}]{hXalphapsit}{\ensuremath{h_X^{\alpha}(\rho, t)}}
\glsxtrnewsymbol[description={a matrix whose determinant equals to $h_X^{\alpha}(\rho, t)$}]{bmHXalphapsit}{\ensuremath{\mathbfit{H}_X^{\alpha}(\rho, t)}}
\glsxtrnewsymbol[description={the matrix associated to voltage assignments $\alpha, \beta$ and to a character $\phi$ of $\mathbb{Z}/p\mathbb{Z}$ with formal power series entries}]{bbWXalphabeta}{\ensuremath{\mathbb{W}_X^{\alpha, \beta}(\phi, T_1,\dotsc, T_d)}}
\glsxtrnewsymbol[description={the formal power series associated to voltage assignments $\alpha,\ \beta$ and to a character $\phi$ of $\mathbb{Z}/p\mathbb{Z}$ with formal series entries}]{QXalphabeta}{\ensuremath{Q_X^{\alpha,\beta}(\phi, T_1,\dotsc, T_d)}}
\glsxtrnewsymbol[description={the generalized Iwasawa $\mu$-invariant of $F$}]{muF}{\ensuremath{\mu(F)}}
\glsxtrnewsymbol[description={the generalized Iwasawa $\lambda$-invariant of $F$}]{lambdaF}{\ensuremath{\lambda(F)}}
\glsxtrnewsymbol[description={the ring of fromal power series over $\mathcal{O}$}]{Lambdad}{\ensuremath{\Lambda_{d,\mathcal{O}}}}
\glsxtrnewsymbol[description={the quotient $\Lambda_{d,\mathcal{O}}/\mathfrak{m}\Lambda_d$}]{Omegad}{\ensuremath{\Omega_{d,\mathcal{O}}}}
\begin{document}
\maketitle
\begin{abstract}
\hspace{4pt}Chung-Langlands established a matrix-tree theorem for positive-real valued vertex-weighted graphs, and Wu-Feng-Sato developed a theory of Ihara zeta functions for those graphs.
In this paper, generalizing and refining these previous works, we initiate the Iwasawa theory for vertex-weighted graphs, which is a generalization of the Iwasawa theory for graphs initiated by Gonet and Valli\`{e}res independently.
First, we generalize the matrix-tree theorem by Chung-Langlands to arbitrary field-valued vertex-weighted graphs.
Second, we refine and prove the so-called decomposition formula for vertex-weighted graphs and edge-weighted graphs without any assumption.
Applying these results, we prove the Iwasawa-type formula and a refinement of Kida's formula for $\mathbb{Z}_p^d$-towers of vertex-weighted graphs.
Our refinement of the decomposition formulas allows us to estimate the root-wise growth of weighted complexities in $\mathbb{Z}_p^d$-towers.
We also provide several numerical examples.
\end{abstract}
\tableofcontents
\section{Introduction}
Let $p$ be a fixed prime number.
Let $\mathbb{Q}_p$ denote the completion of the field of rational numbers $\mathbb{Q}$ equipped with the $p$-adic absolute value, and let $\mathbb{Z}_p$ denote its valuation ring.
Fix an algebraic closure $\overbar{\mathbb{Q}}_p$ of $\mathbb{Q}_p$, and fix an embedding $\mathbb{Q}_p\hookrightarrow \overbar{\mathbb{Q}}_p$.
Let $\val_p$ be the $p$-adic valuation of $\mathbb{Q}_p$ normalized so that $\val_p(p)=1$, and we extend this to the valuation of $\overbar{\mathbb{Q}}_p$ uniquely.
Let $X$ be a finite connected graph with $\overbar{\mathbb{Q}}_p$-valued weights on its edges, and let $\kappa(X)$ denote the edge-weighted complexity of $X$ defined in \cite{MS03}.
In \cite{AMT}, Adachi, Mizuno, and the second author discover that there is an Iwasawa-type formula for $\mathbb{Z}_p^d$-towers of $\overbar{\mathbb{Q}}_p$-valued edge-weighted graphs, that is, 
for a $\mathbb{Z}_p^d$-tower
\begin{equation}\label{eq_34}
X=X_0 \leftarrow X_1 \leftarrow X_2\leftarrow \cdots \leftarrow X_n  \leftarrow \cdots \leftarrow X_{\infty}
\end{equation}
over $X$,
\begin{equation}\label{eq_35}
\val_p(\kappa(X_n))=(\mu p^n+\lambda n)p^{(d-1)n}+\left(\sum_{i=1}^{d-1}(\mu_ip^n+\lambda_in)p^{(d-1-i)n}\right)+\nu
\end{equation}
holds for all $n \gg 0$, where $\lambda$ is an integer and where $\mu$, $\mu_i$, $\lambda_i$, $\nu$ are rational numbers. This is a generalization of the Iwasawa-type formula for graphs initially proved by Gonet \cite{Gon22} and Valli\`{e}res \cite{Val21} independently. See \cite{AMT} for more detailed histories of the Iwasawa theory for graphs.

In this paper, we show that the same theory works for $\overbar{\mathbb{Q}}_p$-valued vertex-weighted graphs, which suggests a possibility that sharpens the results in \cite{AMT}.
The theory of vertex-weighted graphs has been investigated with positive real numbers weights.
For instance, in \cite{MR1401006}, Chung and Langlands prove that a matrix-tree theorem works for $\mathbb{R}_{>0}$-valued weighted graphs, and Wu, Feng, and Sato study some analogue of arithmetical properties in \cite{WFS11}.
A difference between edge-weighted graphs and vertex-weighted graphs is that weighted complexities of vertex-weighted graph heavily depend on the choice of roots, whereas one for ``strongly'' symmetric edge-weighted graphs does not depend on the choice of roots.
The independency of the choice of roots allows us to study them easier, but drops some their information and causes awkward division into cases as in the following theorems (rooted cases in Theorem \ref{thm_10} and Theorem \ref{thm_12}).

In \cite{MR1401006}, Chung and Langlands establish the matrix-tree theorem for positive real-valued vertex-weighted simple graphs. 
To study Iwasawa theory for vertex-weighted graphs, we first generalize their matrix-tree theorem for field-valued vertex weighted possibly non-simple graphs.
Let $K$ be an arbitrary field, and let $X$ be a $K$-valued vertex-weighted connected finite graph.
$V(X)$ stands for the set of the vertices of $X$, and $\mathbb{E}(X)$ stands for the set of the directed edges of $X$. So each edge $e \in \mathbb{E}(X)$ has \emph{origin $o(e)$} and \emph{terminus $t(e)$}. Also, since the graphs we deal with are symmetric directed graphs, there is a unique edge $\overbar{e}$ that represents the opposite direction.
Fix an algebraic closure $\overbar{K}$ of $K$, fix an embedding $K \hookrightarrow \overbar{K}$, and fix square roots of the weights on the vertices of $X$.
For a spanning tree $T$ of $X$ and for a vertex $v$, $T_v$ denotes the rooted tree of $T$ towards $v$, that is, every edge has exactly one direction towards $v$.
The \emph{weight $w(T_v)$ of $T_v$} is $\prod_{e \in \mathbb{E}(T_v)}w_{t(e)}$.
Now we can describe the rooted weighted complexities $\kappa_v(X)$ and the weighted complexity $\kappa(X)$, in which we are mainly interested.
The \emph{rooted (vertex-) weighted complexity $\kappa_v(X)$ at $v$} is $\sum w(T_v)$, where the sum extends over all spanning trees of $X$.
The \emph{(non-rooted vertex-) weighted complexity $\kappa(X)$} is $\sum_{v \in V(X)}\kappa_v(X)$.

Our first result is the matrix-tree theorem for $K$-valued vertex-weighted graphs that are possibly non-simple.
The \emph{degree matrix $\mathbfit{D}_X$} is the diagonal matrix labeled by $V(X)$ whose $(u,u)$- entry is given by 
\begin{equation}
\mathbfit{D}_X(u,u)=\sum_{\substack{e \in \mathbb{E}(X)\colon\\ o(e)=u}}w_{t(e)}.
\end{equation}
The \emph{symmetrized weighted adjacency matrix $\mathcal{W}_X$} is the matrix labeled by $V(X)$ and whose $(u,v)$-entries are given by
\begin{equation}
\quad \mathcal{W}_X(u,v)=\sum_{\substack{e \in \mathbb{E}(X)\colon \\ o(e)=u,\\t(e)=v}}\sqrt{w_u}\sqrt{w_v}.
\end{equation}
The \emph{symmetrized Laplacian matrix $\mathcal{L}_X$ of $X$} is the difference $\mathbfit{D}_X-\mathcal{W}_X$.
We have two kinds of matrix-tree theorems.
One is a formula for rooted complexities.
\begin{theorem}[= Corollary \ref{coro_1}]
Let $X$ be a $K$-valued vertex-weighted connected finite graph.
\begin{equation}
\det(\mathcal{L}_X[\{v\},\{v\}])=\kappa_v(X),
\end{equation}
where $\mathcal{L}_X[\{v\},\{v\}]$ is the minor matrix deleting $v$-th row and $v$-th column.
\end{theorem}
The other is a formula for complexity.
\begin{theorem}[= Theorem \ref{thm_5}]
Let $X$ be a $K$-valued vertex-weighted connected finite graph.
Then 
\begin{equation}
\raisebox{-.75ex}{\bigg(}\sum_{z\in V(X)}w_z\raisebox{-.75ex}{\bigg)}\adj(\mathcal{L}_X)(u,v)=\sqrt{w_u}\sqrt{w_v}\kappa(X),
\end{equation}
where $\adj(\mathcal{L}_X)$ is the classical adjoint matrix of $\mathcal{L}_X$.
\end{theorem}
Let $Y/X$ be a finite Galois cover of edge-weighted graphs.
Then $h$-functions of $Y/X$ mean the three-term-determinant part of Artin-Ihara $L$-functions.
An important property is that the edge-weighted complexity of $Y$ factors into the edge-weighted complexity of $X$ and values of $h$-functions of $Y/X$.
This factorization is originally given in \cite[Corollary 1]{MS03} with the assumption $\left(\sum_{e\in\mathbb{E}(X)}w_e\right)/2-\# V(X)\neq 0$. 
If $X$ is non-weighted, that is, if all the weights $w_e$ are 1, then this assumption means that the Euler characteristic $\chi(X)$ of $X$ is non-zero.
See \cite{AMT} for more details of the background of $h$-functions.
In \cite[Cor.~1]{WFS11}, they prove that, for a finite Galois cover $Y/X$ of vertex-weighted graphs satisfying $\left(\sum_{e\in \mathbb{E}(X)}w_{o(e)}w_{t(e)}\right)/2-\sum_{v \in V(X)}w_v\neq 0$, $\kappa(Y)$ also factors into $\kappa(X)$ and the value of ``weighted'' $h$-functions of $Y/X$.
We prove that the same factorization is valid for \emph{rooted} weighted complexities without the assumption $\sum_{e \in \mathbb{E}(X)}w_e\neq \#V(X)$.
\begin{theorem}[= \eqref{eq_14} in Theorem \ref{thm_6}]\label{thm_7}
Let $Y/X$ be a finite Galois cover of $K$-valued vertex-weighted graphs whose Galois group is $G$, and let $v$ be a vertex of $X$.
Let $\widehat{G}$ be a full set of representatives of equivalence classes for the irreducible representations of $G$. 
Then, for every vertex $w$ of $Y$ belonging to the fiber of $v$,
\begin{equation}\label{eq_36}
\kappa_w(Y)=\frac{\kappa_v(X)}{\# G}\prod_{\rho \in \widehat{G} \setminus \{\triv_G\}}(h_{Y/X}(\rho, 1))^{d_{\rho}},
\end{equation}
where $\triv_G$ is the trivial representation of $G$ and where $d_{\rho}$ is the degree of $\rho$.
In particular, $\kappa_w(Y)$ is constant on the fiber of $v$, so we would rather write $\kappa_v(Y)$ than $\kappa_w(Y)$ for $w$ belonging to the fiber of $v$.
\end{theorem}
In \cite{AMT}, the same type of Theorem \ref{thm_7} is shown assuming that $G$ is abelian.
The same proof in our theorem is also valid for edge-weighted graph theory, and Theorem \ref{thm_7} gives an improvement of their result.
We remark also that our proof is a simple calculation of determinants of matrices, so we do not need the assumption that $\big(\sum_{e\in\mathbb{E}(X)}w_{o(e)}w_{t(e)}\big)/2-\sum_{v\in V(X)}w_v\neq 0$.
Hence our theorem sharpens \cite[Corollary 1]{MS03}.

Summing up all \eqref{eq_36}, we can deduce the factorization for weighted complexities in \cite[Cor.~1]{WFS11} without the assumption $\sum_{v \in V(X)}w_v\neq \#V(X)$, which is tacitly assumed in \cite{WFS11}.
\begin{theorem}[= Theorem \ref{thm_6} \eqref{eq_21}]\label{thm_8}
Let $Y/X$ be a finite Galois cover of $K$-valued vertex-weighted graphs whose Galois group is $G$.
Let $\widehat{G}$ be a full set of representatives of equivalence classes for the irreducible representations of $G$. 
Then
\begin{equation}\label{eq_37}
\kappa(Y)=\kappa(X)\prod_{\rho \in \widehat{G} \setminus \{\triv_G\}}\big(h_{Y/X}(\rho, 1)\big)^{d_{\rho}}.
\end{equation}
\end{theorem}
Note that \eqref{eq_36} involves $\#G$ as the denominator of right-hand side, but \eqref{eq_37} does not.
This fact affects the Iwasawa $\lambda$-invariant for \eqref{eq_35}.

Now we embark on the Iwasawa theory for vertex-weighted graphs.
Let $K$ be a finite extension of $\mathbb{Q}_p$, let $\mathcal{O}$ be the valuation ring of $K$, and let $\Lambda_{d, \mathcal{O}}$ be the formal power series ring $\mathcal{O}\lBrack T_1,\dotsc, T_d\rBrack$.
Then, for an element $F$ of $\Lambda_{d,\mathcal{O}}\otimes_{\mathcal{O}}K$, we can extract a rational number $\mu(F)$ and a integer $\lambda(F)$ that are called Iwasawa $\mu$-invariant and $\lambda$-invariant of $F$ respectively.
A fundamental result of multivariable Iwasawa theory is the following.
\begin{theorem}[{\cite[Theorem 3.5]{AMT}}]\label{thm_9}
Put $W=\{\,\zeta \in \overbar{\mathbb{Q}}_p\;|\; \text{$\zeta^{p^n}=1$ for some non-negative integer $n$}\,\}$, and write 
$W^d(n)=\{\,\mathbfit{\zeta}\in W^d\;|\; \mathbfit{\zeta}^{p^n}=\mathbfit{1}\,\}$.
Let $F$ be an non-zero element of $\Lambda_{d, \mathcal{O}}\otimes_{\mathcal{O}}K$.
Assume that the function $\zeta\mapsto F(\mathbfit{\zeta}-\mathbfit{1})$ does not vanish on $W^d$.
Then there exist $\mu_1,\dotsc,\ \mu_{d-1},\ \lambda_1,\dotsc,\ \lambda_{d-1},\nu \in \mathbb{Q}$ such that 
\begin{equation}\label{eq_38}
\sum_{\mathbfit{\zeta}\in W^d(n) \setminus \{\mathbfit{1}\}}\val_p(F(\mathbfit{\zeta}-\mathbfit{1}))=(\mu(F) p^n+\lambda(F) n)p^{(d-1)n}+\left(\sum_{i=1}^{d-1}(\mu_i p^n+\lambda n)p^{(d-1-i)n}\right)+\nu
\end{equation}
for all $n \gg 0$.
\end{theorem}
On the other hand, for a $\mathbb{Z}_p^d$-tower over $X$, there is an element $Q_{X^{\infty}/X}$ of $\Lambda_{d, \mathcal{O}}\otimes_{\mathcal{O}}K$ such that $Q_{X^{\infty}/X}(\mathbfit{\zeta}-\mathbfit{1})=h_{X_n/X}(\mathbfit{\zeta},1)$ for all $\mathbfit{\zeta} \in W^d(n) = ((\mathbb{Z}/p^n\mathbb{Z})^d)\hat{\hphantom{)}}$, and $Q_{X^{\infty}/X}$ satisfies the assumption of Theorem \ref{thm_9} if $\kappa_v(X_n) \neq 0$ for all $n$ and for fixed $v \in V(X)$ (or equivalently $\kappa(X_n) \neq 0$ for all $n$).
Therefore, combining \eqref{eq_36} and \eqref{eq_38}, we have the following result.
\begin{theorem}[= Theorem \ref{thm_2}]\label{thm_10}
Let $X_{\infty}/X$ be a $\mathbb{Z}_p^d$-tower over a $K$-valued vertex-weighted graph $X$ such that $X_n$ is connected for all $n$, and $\kappa_v(X_n) \neq 0$ for all $n$ and for fixed $v \in V(X)$. Then there exist $\mu_1,\dotsc,\ \mu_{d-1},\ \lambda_1,\dotsc,\ \lambda_{d-1},\nu \in \mathbb{Q}$ such that 
\begin{equation}\label{eq_39}
\val_p(\kappa_v(X_n))=(\mu_v(X_{\infty}/X) p^n+\lambda_v(X_{\infty}/X) n)p^{(d-1)n}+\left(\sum_{i=1}^{d-1}(\mu_i p^n+\lambda n)p^{(d-1-i)n}\right)+\nu
\end{equation}
for all $n \gg 0$, where $\mu_v(X_{\infty}/X)=\mu(Q_{X_{\infty}/X})$ and where 
\begin{equation}\label{eq_40}
\lambda_v(X_{\infty}/X)=
\begin{cases}
\lambda(Q_{X_{\infty}/X})-1 & \text{if $d=1$,} \\
\hfill \lambda(Q_{X_{\infty}/X}) \hfill\hfill & \text{if $d\geq 2$.}
\end{cases}
\end{equation}
\end{theorem}
Theorem \ref{thm_10} is the vertex-weighted analogue of \cite[Theorem 6.2]{DV23}, \cite[Theorem 4.3]{KM24}, and \cite[Theorem 3.9]{AMT}.
The division into cases caused in \eqref{eq_40} (and in \cite[Theorem 3.9]{AMT}) is due to the denominator $\#G$ of the right-hand side in \eqref{eq_36}.
Keeping this remark in mind, we also have another result, which gives more natural interpretation for the Iwasawa-type formula for vertex-weighted graphs.
\begin{theorem}[= Theorem \ref{thm_3}]\label{thm_11}
Let $X_{\infty}/X$ be a $\mathbb{Z}_p^d$-tower over a $K$-valued vertex-weighted graph $X$ such that $X_n$ is connected for all $n$, and $\kappa(X_n) \neq 0$ for all $n$. Then there exist $\mu_1,\dotsc,\ \mu_{d-1},\ \lambda_1,\dotsc,\ \lambda_{d-1},\nu \in \mathbb{Q}$ such that 
\begin{equation}\label{eq_39}
\val_p(\kappa(X_n))=(\mu(X_{\infty}/X) p^n+\lambda(X_{\infty}/X) n)p^{(d-1)n}+\left(\sum_{i=1}^{d-1}(\mu_i p^n+\lambda n)p^{(d-1-i)n}\right)+\nu
\end{equation}
for all $n \gg 0$, where $\mu(X_{\infty}/X)=\mu(Q_{X_{\infty}/X}),\ \lambda(X_{\infty}/X)=\lambda(Q_{X_{\infty}/X})$.
\end{theorem}

We also prove an analogue of Kida's formula for vertex-weighted graphs with additional formula for Iwasawa $\mu$-invariants.
\begin{theorem}[= Theorem \ref{thm_40}]\label{thm_12}
Let $Y/X$ be a Galois cover of $K$-valued vertex-weighted graphs whose Galois group is of $p$ power order.
Let $X_{\infty}/X$ and $Y_{\infty}/Y$ be $\mathbb{Z}_p^d$-towers such that Galois group of $Y_n/X$ is $\Gal(Y/X) \times (\mathbb{Z}/p^n\mathbb{Z})^d$ for all $n$.
If $(a)$, either $(b)$ or $(b)'$, and either $(c)$ or $(c)'$ 
\begin{enumerate}[leftmargin=39pt]
\item[$(a)\phantom{'}$] All the $Y_n$ are connected.
\item[$(b)\phantom{'}$] All the $\kappa_v(Y_n)$ are non-zero for some fixed $v \in V(X)$. 
\item[$(b)'$] All the $\kappa(Y_n)$ are non-zero.
\item[$(c)\phantom{'}$] $\val_p(w_v) \geq \mu(X_{\infty}/X)/\#V(X) $ for every $v \in V(X)$.
\item[$(c)'$] $\val_p(w_v) \geq \mu(Y_{\infty}/Y)/\#V(Y) $ for every $v \in V(Y)$.
\end{enumerate}
are satisfied, then the following $(1)$ and $(2)$ hold.
\begin{enumerate}[label=$(\arabic*)$]
\item 
 $\mu_v(Y_{\infty}/Y) = [Y:X] \mu_v(X_{\infty}/X)$ for each $v \in V(X)$, and $\mu(Y_{\infty}/Y)=[Y:X]\mu(X_{\infty}/X)$.
\item
\[
\lambda_v(Y_{\infty}/Y)=
\begin{cases}
\hfill [Y:X]\lambda_v(X_{\infty}/X) \hfill\hfill& \text{if $d \geq 2$,} \\
[Y:X](\lambda_v(X_{\infty}/X)+1)-1 & \text{if $d=1$}.
\end{cases}
\]
for each $v \in V(X)$, and $\lambda(Y_{\infty}/Y)=[Y:X]\lambda(X_{\infty}/X)$.
\end{enumerate}
\end{theorem}
Again, the formula for $\lambda$-invariants on rooted weighted complexities divides into cases, but the one for $\lambda$-invariants on weighted-complexities does not.
Note that if we apply these formulas to non-weighted graphs (that can be considered as vertex-weighted graphs whose weights are all 1), assumption $(c)$ and $(c)'$ are equivalent to the condition that $\mu(X_{\infty}/X)=0$ if and only of $\mu(Y_{\infty}/Y)$ thanks to (1) in Theorem \ref{thm_12}. Hence Theorem \ref{thm_12} is a direct generalization of \cite[Theorem 4.1]{RV22}.
\subsection*{Organization of the paper}
In \S \ref{sec_2}, we prepare basic notions for our proofs including the matrix-tree theorem for vertex-weighted graphs.
In \S \ref{subsec_2.1}, we introduce the Laplacian matrices for a field-valued vertex-weighted graphs that is possibly non-simple, and we give the matrix-tree theorem for rooted weighted complexities and weighted complexities.
In \S \ref{subsec_2.2}, we give some explanation of Galois theory of vertex-weighted graphs using the language of voltage assignments.

In \S \ref{sec_3}, we give the proof of the formula of weighted complexities of Galois cover $Y/X$ with $h$-functions, which is the technical heart of our proofs.
In \S \ref{subsec_3.1}, we show fundamental properties of $h$-function, namely direct sum rules and induction properties by a simple calculation of matrices.

In \S \ref{sec_4}, we prove the first main theorem.
In \S \ref{subsec_4.1}, we recall the fundamental theorem of multivariable Iwasawa theory according to \cite[Sec.~3.2]{AMT}.
Then we prove an Iwasawa-type formula for rooted weighted complexities and weighted complexities in \S \ref{subsec_4.2}.

In \S \ref{sec_5}, we give a proof of Kida's formula for vertex-weighted graphs, which is our second main theorem.

\subsection*{Acknowledgements}
The authors would like to thank Taiga Adachi, Kosuke Mizuno, Iwao Sato, and the anonymous referee for valuable comments. The first author was supported by Make New Standards Program for the Next Generation Researchers.
The second author was supported by JSPS KAKENHI, Grant Number 26KJ0138.
\section{Vertex-weighted graph theory}\label{sec_2}
We recall the definitions of graphs in the following \cite{AMT}.
\begin{definition}\ 
\begin{enumerate}[label=(\roman*)]
\item
A \emph{directed graph $X$} is a triple $(V(X), \mathbb{E}(X), \inc_X)$, where \glssymbol{VX}$V(X)$ and \glssymbol{bbEX}$\mathbb{E}(X)$ are sets and where $\inc_X$ is a map from $\mathbb{E}(X)$ to $V(X) \times V(X)$.
The elements of $V(X)$ are called the \emph{vertices of $X$}, the elements of $\mathbb{E}(X)$ are called the \emph{directed edges of $X$}, and $\inc_X$ is called the \emph{incident map}.
We write $(o(e), t(e))$ for the image of an edge $e$ under $\inc_X$, \glssymbol{oe}$o(e)$ is called the \emph{origin of $e$}, and \glssymbol{te}$t(e)$ is called the \emph{terminus of $e$}.
\item
An directed graph $X$ is called \emph{finite} if $V(X)$ and $\mathbb{E}(X)$ are finite sets.
\item 
A \emph{symmetric directed graph $X$} is a directed graph equipped with a map $e \mapsto \overbar{e}$ of $\mathbb{E}(X)$ to itself such that $\overbar{e}\neq e$, $\overbar{\overbar{e}}=e$, and $o(e)=t(\overbar{e})$.
The directed edge \glssymbol{bare}$\overbar{e}$ is called the \emph{inverse directed edge of an edge $e$}.
\item A \emph{subgraph $Y$ of a directed graph $X$} is a directed graph requiring that $V(Y) \subseteq V(X)$, $\mathbb{E}(Y) \subseteq \mathbb{E}(X)$, and $\inc_X|_{\mathbb{E}(Y)}=\inc_Y$.
\end{enumerate}
\end{definition}
\begin{remark}
For a directed graph $X$, consider the quotient set $E(X)$ of $\mathbb{E}(X)$ under the equivalence relation generated by $e \sim \overbar{e}$.
Assigning an equivalence class $[e]$ to $\{o(e), t(e)\}$, we obtain the \emph{incidence map} $\inc_{X_{\mathrm{undir}}}\colon E(X) \to 2^{V(X)}$.
Then the triple \glssymbol{Xundir}$X_{\mathrm{undir}}=\left(V(X), E(X), \inc_{X_{\mathrm{undir}}}\right)$ is a usual undirected graph, and the assignment $X \mapsto X_{\mathrm{undir}}$ from the class of symmetric directed graphs to the class of undirected graphs induces an equivalence of categories.
Henceforth we identify the symmetric directed graphs as the undirected graphs.
\end{remark}
We want to consider spanning trees of a symmetric directed graph.
To do this we first recall the notion of paths and cycles.
\begin{definition}
Let $X$ be a symmetric directed graph.
\begin{enumerate}[label=(\roman*)]
\item A \emph{walk from $u$ to $v$} is a finite sequence of edges $(e_1,e_2,\dotsc, e_l)$ imposing the following condition: $o(e_1)=u$, $t(e_l)=v$, and $t(e_i)=o(e_{i+1})$ for all $i \in \{1,2,\dotsc, l-1\}$. 
\item
A \emph{trail from $u$ to $v$} is a walk $(e_1,e_2,\dotsc, e_l)$ from $u$ to $v$  imposing the following condition: $\{e_i, \overbar{e}_i\} \neq \{e_j, \overbar{e}_j\}$ for distinct $i, j \in \{1,2,\dotsc, l\}$.
\item
A \emph{path from $u$ to $v$} is a trail $(e_1,e_2,\dotsc, e_l)$ from $u$ to $v$ imposing the following condition: $\#\bigcup_{i=1}^l \{o(e_i), t(e_i)\} = l+1$.
\item
A \emph{cycle} is a trail $(e_1,e_2,\dotsc, e_l)$ imposing the following conditions: $o(e_1)=t(e_l)$, and\linebreak $\#\bigcup_{i=1}^l \left\{o(e_i), t(e_i)\right\} = l$.
\end{enumerate}
\end{definition}
\begin{definition}\ 
\begin{enumerate}[label=(\roman*)]
\item A symmetric directed graph is said to be \emph{connected} if there is a path from $u$ to $v$ for every distinct $u,\ v \in V(X)$. 
\item
A \emph{forest $F$} is a symmetric directed graph that has no cycle.
\item 
A \emph{tree $T$} is a connected forest.
Note that there exists a unique path from $u$ to $v$ for every distinct $u,\ v \in V(X)$, and
the number \big(written \glssymbol{dTuv}$d_T(u, v)$\big) of edges consisting of the path is called the \emph{distance between $u$ and $v$ in $T$}
\item 
A \emph{spanning tree $T$ of a connected symmetric directed graph $X$} is a subtree of $X$ with the property that $V(T) = V(X)$.
\item 
For a tree $T$ and for a vertex $v \in V(T)$, the \emph{rooted tree \glssymbol{Tv}$T_v$ toward $v$} is a directed subgraph whose edge set is given by $\mathbb{E}(T_v)=\{\,e \in \mathbb{E}(T) \;|\; d_T(v, o(e))>d_T(v, t(e))\,\}$.
\end{enumerate}
\end{definition}
\begin{example}\label{ex_1}
We illustrate a symmetric directed graph $X$ in Figure \ref{fig_1} with only one directions.
For instance, we omit illustrating the inverse edge $\overbar{e}_1$ of $e_1$.
The following (1)--(4) are examples of walks, trails, paths, {\and} cycles.
\begin{enumerate}[label=(\arabic*)]
\item 
$C_1=(e_1, e_2, \overbar{e}_2, e_5)$ is a walk from $v_1$ to $v_4$, but it is not a trail.
\item 
$C_2=(e_1, e_2, \overbar{e}_3, e_5)$ is a trail from $v_1$ to $v_4$, but it is not a path.
\item 
$C_3=(e_1, e_5)$ is a path from $v_1$ to $v_4$.
\item 
$C_4=(e_6)$, $C_5=(e_2, \overbar{e}_3)$ and $C_6=(e_1,e_2,e_4)$ are cycles.
\end{enumerate}
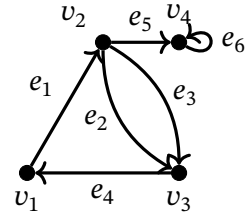
\begin{figure}[t]
\centering
\begin{tikzpicture}[scale=1]
\coordinate[label=below:$v_1$,draw, circle, inner sep=1pt, fill=black](v1)at(0,0);
\coordinate[label=above left:$v_2$,draw, circle, inner sep=1pt, fill=black](v2)at({2*cos(pi/3 r)}, {2*sin(pi/3 r)});
\coordinate[label=below:$v_3$,draw, circle, inner sep=1pt, fill=black](v3)at(2,0);
\coordinate[label=above:$v_4$,draw, circle, inner sep=1pt, fill=black](v4)at(2,{2*sin(pi/3 r)});

\draw(v1)[->, thick] to  node[above left]{$e_1$} (v2);
\draw(v3)[->, thick] to  node[below]{$e_4$} (v1);
\draw(v2)[->, thick]  to [bend left]  node[right]{$e_3$} (v3);
\draw(v2)[->, thick]  to [bend right]  node[left]{$e_2$} (v3);
\draw(v2)[->, thick] to  node[above]{$e_5$} (v4);
\draw(v4)[->, thick]to[out=330, in=30, loop] node[right]{$e_6$} ();
\end{tikzpicture}
\caption{symmetric directed graph $X$}\label{fig_1}
\end{figure}
Figure \ref{fig_2} illustrates all the spanning trees of $X$. (Note that the $T_i$ are symmetric directed graphs.)
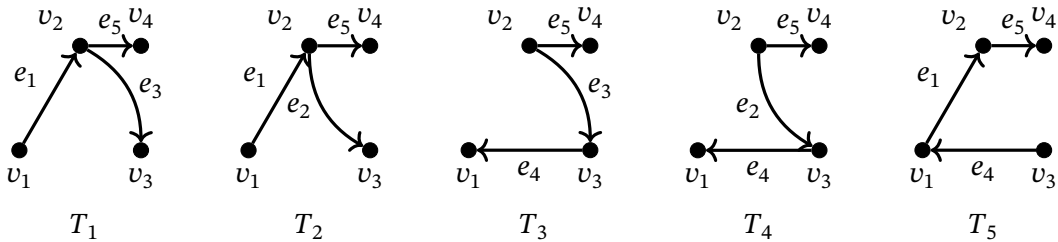
\begin{figure}[tb]
\begin{minipage}{0.18\columnwidth}
\centering
\begin{tikzpicture}[scale=0.8]
\coordinate[label=below:$v_1$,draw, circle, inner sep=1pt, fill=black](v1)at(0,0);
\coordinate[label=above left:$v_2$,draw, circle, inner sep=1pt, fill=black](v2)at({2*cos(pi/3 r)}, {2*sin(pi/3 r)});
\coordinate[label=below:$v_3$,draw, circle, inner sep=1pt, fill=black](v3)at(2,0);
\coordinate[label=above:$v_4$,draw, circle, inner sep=1pt, fill=black](v4)at(2,{2*sin(pi/3 r)});

\draw(v1)[->, thick] to  node[above left]{$e_1$} (v2);
\draw(v2)[->, thick]  to [bend left]  node[right]{$e_3$} (v3);
\draw(v2)[->, thick] to  node[above]{$e_5$} (v4);
\end{tikzpicture}\\
$T_1$
\end{minipage}
\begin{minipage}{0.18\columnwidth}
\centering
\begin{tikzpicture}[scale=0.8]
\coordinate[label=below:$v_1$,draw, circle, inner sep=1pt, fill=black](v1)at(0,0);
\coordinate[label=above left:$v_2$,draw, circle, inner sep=1pt, fill=black](v2)at({2*cos(pi/3 r)}, {2*sin(pi/3 r)});
\coordinate[label=below:$v_3$,draw, circle, inner sep=1pt, fill=black](v3)at(2,0);
\coordinate[label=above:$v_4$,draw, circle, inner sep=1pt, fill=black](v4)at(2,{2*sin(pi/3 r)});

\draw(v1)[->, thick] to  node[above left]{$e_1$} (v2);
\draw(v2)[->, thick]  to [bend right]  node[left]{$e_2$} (v3);
\draw(v2)[->, thick] to  node[above]{$e_5$} (v4);
\end{tikzpicture}\\
$T_2$
\end{minipage}
\begin{minipage}{0.18\columnwidth}
\centering
\begin{tikzpicture}[scale=0.8]
\coordinate[label=below:$v_1$,draw, circle, inner sep=1pt, fill=black](v1)at(0,0);
\coordinate[label=above left:$v_2$,draw, circle, inner sep=1pt, fill=black](v2)at({2*cos(pi/3 r)}, {2*sin(pi/3 r)});
\coordinate[label=below:$v_3$,draw, circle, inner sep=1pt, fill=black](v3)at(2,0);
\coordinate[label=above:$v_4$,draw, circle, inner sep=1pt, fill=black](v4)at(2,{2*sin(pi/3 r)});

\draw(v3)[->, thick] to  node[below]{$e_4$} (v1);
\draw(v2)[->, thick]  to [bend left]  node[right]{$e_3$} (v3);
\draw(v2)[->, thick] to  node[above]{$e_5$} (v4);
\end{tikzpicture}\\
$T_3$
\end{minipage}
\begin{minipage}{0.18\columnwidth}
\centering
\begin{tikzpicture}[scale=0.8]
\coordinate[label=below:$v_1$,draw, circle, inner sep=1pt, fill=black](v1)at(0,0);
\coordinate[label=above left:$v_2$,draw, circle, inner sep=1pt, fill=black](v2)at({2*cos(pi/3 r)}, {2*sin(pi/3 r)});
\coordinate[label=below:$v_3$,draw, circle, inner sep=1pt, fill=black](v3)at(2,0);
\coordinate[label=above:$v_4$,draw, circle, inner sep=1pt, fill=black](v4)at(2,{2*sin(pi/3 r)});

\draw(v3)[->, thick] to  node[below]{$e_4$} (v1);
\draw(v2)[->, thick]  to [bend right]  node[left]{$e_2$} (v3);
\draw(v2)[->, thick] to  node[above]{$e_5$} (v4);
\end{tikzpicture}\\
$T_4$
\end{minipage}
\begin{minipage}{0.18\columnwidth}
\centering
\begin{tikzpicture}[scale=0.8]
\coordinate[label=below:$v_1$,draw, circle, inner sep=1pt, fill=black](v1)at(0,0);
\coordinate[label=above left:$v_2$,draw, circle, inner sep=1pt, fill=black](v2)at({2*cos(pi/3 r)}, {2*sin(pi/3 r)});
\coordinate[label=below:$v_3$,draw, circle, inner sep=1pt, fill=black](v3)at(2,0);
\coordinate[label=above:$v_4$,draw, circle, inner sep=1pt, fill=black](v4)at(2,{2*sin(pi/3 r)});

\draw(v1)[->, thick] to  node[above left]{$e_1$} (v2);
\draw(v3)[->, thick] to  node[below]{$e_4$} (v1);
\draw(v2)[->, thick] to  node[above]{$e_5$} (v4);
\end{tikzpicture}\\
$T_5$
\end{minipage}
\caption{five spanning trees of $X$}\label{fig_2}
\end{figure}
Figure \ref{fig_3} illustrates all the rooted trees of $T_1$.
(Note that the $T_i$ are not symmetric directed graphs.)
\begin{figure}
\begin{minipage}{0.22\columnwidth}
\centering
\begin{tikzpicture}[scale=0.8]
\coordinate[label=below:$v_1$,draw, circle, inner sep=1pt, fill=black](v1)at(0,0);
\coordinate[label=above left:$v_2$,draw, circle, inner sep=1pt, fill=black](v2)at({2*cos(pi/3 r)}, {2*sin(pi/3 r)});
\coordinate[label=below:$v_3$,draw, circle, inner sep=1pt, fill=black](v3)at(2,0);
\coordinate[label=above:$v_4$,draw, circle, inner sep=1pt, fill=black](v4)at(2,{2*sin(pi/3 r)});

\draw(v2)[->, thick] to  node[above left]{$\overbar{e}_1$} (v1);
\draw(v3)[->, thick]  to [bend right]  node[right]{$\overbar{e}_3$} (v2);
\draw(v4)[thick] to  node[above]{$\overbar{e}_5$} (v2);
\draw(v4)[->, thick, shorten > = 4pt] to  (v2);
\end{tikzpicture}\\
$(T_1)_{v_1}$
\end{minipage}
\begin{minipage}{0.22\columnwidth}
\centering
\begin{tikzpicture}[scale=0.8]
\coordinate[label=below:$v_1$,draw, circle, inner sep=1pt, fill=black](v1)at(0,0);
\coordinate[label=above left:$v_2$,draw, circle, inner sep=1pt, fill=black](v2)at({2*cos(pi/3 r)}, {2*sin(pi/3 r)});
\coordinate[label=below:$v_3$,draw, circle, inner sep=1pt, fill=black](v3)at(2,0);
\coordinate[label=above:$v_4$,draw, circle, inner sep=1pt, fill=black](v4)at(2,{2*sin(pi/3 r)});

\draw(v1)[->, thick] to  node[above left]{$e_1$} (v2);
\draw(v3)[->, thick]  to [bend right]  node[right]{$\overbar{e}_3$} (v2);
\draw(v4)[thick] to  node[above]{$\overbar{e}_5$} (v2);
\draw(v4)[->, thick, shorten >=4pt] to (v2);
\end{tikzpicture}\\
$(T_1)_{v_2}$
\end{minipage}
\begin{minipage}{0.22\columnwidth}
\centering
\begin{tikzpicture}[scale=0.8]
\coordinate[label=below:$v_1$,draw, circle, inner sep=1pt, fill=black](v1)at(0,0);
\coordinate[label=above left:$v_2$,draw, circle, inner sep=1pt, fill=black](v2)at({2*cos(pi/3 r)}, {2*sin(pi/3 r)});
\coordinate[label=below:$v_3$,draw, circle, inner sep=1pt, fill=black](v3)at(2,0);
\coordinate[label=above:$v_4$,draw, circle, inner sep=1pt, fill=black](v4)at(2,{2*sin(pi/3 r)});

\draw(v1)[->, thick] to  node[above left]{$e_1$} (v2);
\draw(v2)[->, thick]  to [bend left]  node[right]{$e_3$} (v3);
\draw(v4)[->, thick] to  node[above]{$\overbar{e}_5$} (v2);
\end{tikzpicture}\\
$(T_1)_{v_3}$
\end{minipage}
\begin{minipage}{0.22\columnwidth}
\centering
\begin{tikzpicture}[scale=0.8]
\coordinate[label=below:$v_1$,draw, circle, inner sep=1pt, fill=black](v1)at(0,0);
\coordinate[label=above left:$v_2$,draw, circle, inner sep=1pt, fill=black](v2)at({2*cos(pi/3 r)}, {2*sin(pi/3 r)});
\coordinate[label=below:$v_3$,draw, circle, inner sep=1pt, fill=black](v3)at(2,0);
\coordinate[label=above:$v_4$,draw, circle, inner sep=1pt, fill=black](v4)at(2,{2*sin(pi/3 r)});

\draw(v1)[->, thick, shorten >=5pt] to (v2);
\draw(v1)[thick] to  node[above left]{$e_1$} (v2);
\draw(v3)[->, thick]  to [bend left]  node[right]{$\overbar{e}_3$} (v2);
\draw(v2)[->, thick] to  node[above]{$e_5$} (v4);
\end{tikzpicture}\\
$(T_1)_{v_4}$
\end{minipage}
\caption{four rooted trees of $T_1$}\label{fig_3}
\end{figure}
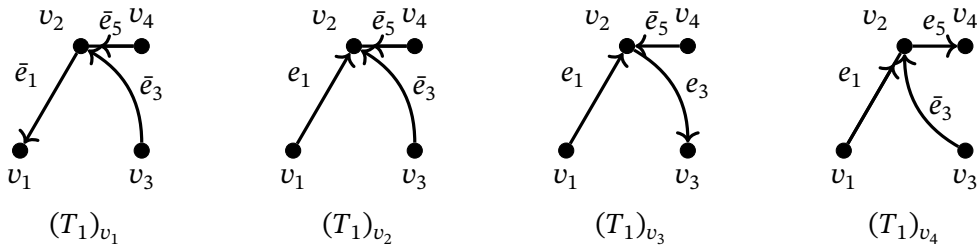
\end{example}
\subsection{Vertex-weighted matrix-tree theorem}\label{subsec_2.1}
The materials in this subsection is based on \cite{MR1401006}, but we slightly generalize their results.
Concretely speaking, we show that the vertex-weighted matrix-tree theorem works on some field-valued weights and on possibly non-simple graphs.
\begin{definition}
Let $K$ be a field.
A \emph{$K$-valued vertex-weighted graph} is a symmetric directed graph $X$ equipped with a ``weight'' function $w\colon V(X) \to K$.
We write \glssymbol{wv}$w_v$ for the weight of $v \in V(X)$.
\end{definition}
For a symmetric directed finite graph $X$ with no vertex-weights nor edge-weights, 
the \emph{degree matrix $\mathbfit{D}_X$ of $X$} is a diagonal matrix labeled by the vertices of $X$ and whose $(u,u)$-entry is $\#\{\,e \in \mathbb{E}(X)\;|\;o(e)=u\,\}$, the \emph{adjacency matrix $\mathbfit{A}_X$ of $X$} is a square matrix labeled by the vertices of $X$ and whose $(u,v)$-entry is $\#\{\,e \in \mathbb{E}(X) \;|\; o(e)=u,\ t(e)=v\,\}$, and
the \emph{Laplacian matrix $\mathbfit{L}_X$ of $X$} is the difference
\begin{equation}\label{eq_5}
\mathbfit{L}_X=\mathbfit{D}_X-\mathbfit{A}_X.
\end{equation}
Here these matrices can be considered as usual matrices using a fixed total order on $V(X)$, and our results do not depend on the choice of every total order.
Using this numbering, products of these kinds of matrices can be calculated.

The Laplacian matrix of a vertex-weighted finite graph is defined in the following, way that naturally generalizes usual Laplacian matrices of symmetric directed finite graphs.
\begin{definitions}
Let $K$ be a field, and let $X$ be a $K$-valued vertex-weighted finite graph.
\begin{enumerate}[label=(\roman*)]
\item 
The \emph{degree matrix \glssymbol{bfD}$\mathbfit{D}_X$ of $X$} is a diagonal matrix labeled by the vertices of $X$ and whose $(u,u)$ entry is 
\begin{equation}
\mathbfit{D}_X(u,u)=\sum_{\substack{e \in \mathbb{E}(X)\colon o(e)=u}}w_{t(e)}.
\end{equation}
\item 
The \emph{weighted adjacency matrix \glssymbol{bfW}$\mathbfit{D}_X$ of $X$} is a square matrix labeled by the vertices of $X$ and whose $(u,v)$ entry is
\begin{equation}
\mathbfit{W}_X(u,v)=\sum_{\substack{e \in \mathbb{E}(X)\colon\\ o(e)=u, \\ t(e)=v}}w_{v}.
\end{equation}
\item
The \emph{Laplacian matrix \glssymbol{bfL}$\mathbfit{L}_X$ of $X$} is the difference
\begin{equation}\label{eq_6}
\mathbfit{L}_X= \mathbfit{D}_X-\mathbfit{W}_X.
\end{equation}
\end{enumerate}
\end{definitions}
\begin{remark}
When each vertex has weight 1, the Laplacian matrix in \eqref{eq_6} is the same as the Laplacian matrix in \eqref{eq_5}.
\end{remark}
Although $\mathbfit{W}_X$ and hence $\mathbfit{L}_X$ may not be symmetric, one can modify $\mathbfit{L}_X$ to a symmetric matrix.
\begin{definition}
Let $K$ be a field, and let $X$ be a $K$-valued vertex-weighted finite graph.
Fix an algebraic closure $\overbar{K}$, fix an embedding $K \hookrightarrow \overbar{K}$, fix a square root $\sqrt{w_v}$ for each $v \in V(X)$, and set $K'=K\left(\left\{\sqrt{w_v}\;|\; v \in V(X)\right\}\right)$. 
Then we define the \emph{symmetrized weighted adjacency matrix} \glssymbol{calW}$\mathcal{W}_X$ over $K'$ by setting the $(u,v)$-entry as 
\begin{equation}
\mathcal{W}_X(u,v)=\sum_{\substack{e \in \mathbb{E}(X)\colon \\ o(e)=u, \\ t(e)=v}}\sqrt{w_u}\sqrt{w_v}.
\end{equation}
The \emph{symmetrized Laplacian matrix \glssymbol{calL}$\mathcal{L}_X$} is the difference
\begin{equation}
\mathcal{L}_X=\mathbfit{D}_X-\mathcal{W}_X.
\end{equation}
\end{definition}
\begin{proposition}\label{prop_1}
Let $K$ be a field, let $X$ be a vertex $K$-valued weighted finite graph, and let \glssymbol{scrW}$\mathscr{W}_X$ be a diagonal matrix whose $(v,v)$-entry is $w_v$.
Note that $\mathscr{W}_X$ has a square root, and let $\sqrt{\mathscr{W}_X}$ denote the diagonal matrix whose $(v,v)$-entry is $\sqrt{w_v}$.
Then
\begin{equation}
\sqrt{\mathscr{W}_X}\mathbfit{W}_X= \mathcal{W}_X\sqrt{\mathscr{W}_X},
\end{equation}
and hence
\begin{equation}\label{eq_1}
\sqrt{\mathscr{W}_X}\mathbfit{L}_X= \mathcal{L}_X\sqrt{\mathscr{W}_X}.
\end{equation}
\end{proposition}
\begin{proof}
For $u,\ v \in V(X)$, we can calculate
\begin{align*}
\left(\sqrt{\mathscr{W}_X}\mathbfit{W}_X\right)(u,v)
&=\sum_{\substack{e \in \mathbb{E}(X)\colon \\ o(e)=u, \\ t(e)=v}}\sqrt{w_u}w_v\\
&=\left(\mathcal{W}_X\sqrt{\mathscr{W}_X}\right)(u,v).
\end{align*}
This means that $\sqrt{\mathscr{W}_X}\mathbfit{L}_X= \mathcal{L}_X\sqrt{\mathscr{W}_X}$.
\end{proof}
\begin{remark}
When all the weights of the vertices are non-zero, $\sqrt{\mathscr{W}_X}$ is invertible, and hence $\mathbfit{L}_X=\left(\sqrt{\mathscr{W}_X}\right)^{-1}\mathcal{L}_X\sqrt{\mathscr{W}_X}$, that is, $\mathbfit{L}_X$ and $\mathcal{L}_X$ are equivalent.
\end{remark}Since two kinds of Laplacian matrices $\mathbfit{L}_X$ and $\mathcal{L}_X$ can be considered as ``equivalent matrices,'' we will investigate $\mathcal{L}_X$ instead of $\mathbfit{L}_X$.
\begin{corollary}
For an arbitrary $K$-valued vertex-weighted finite graph $X$, $\mathcal{L}_X\left(\sqrt{\mathscr{W}_X}\mathbfit{1}\right)=\mathbfit{0}$, where $\mathbfit{1}$ denotes the column vector all of whose entries are 1.
Hence $\rank\mathcal{L}_X \leq \#V(X)-1$.
\end{corollary}
\begin{proof}
By \eqref{eq_1}, 
$\mathcal{L}_X\left(\sqrt{\mathscr{W}_X}\mathbfit{1}\right)=\sqrt{\mathscr{W}_X}(\mathbfit{L}_X\mathbfit{1})$. For each $u \in V(X)$,
\[
(\mathbfit{L}_X\mathbfit{1})(u)=\sum_{\substack{e \in \mathbb{E}(X)\colon \\ o(e)=u}}w_{t(e)}-\sum_{v \in V(G)}\sum_{\substack{e \in \mathbb{E}(X)\colon \\ o(e)=u, \\ t(e)=v}}w_v=0.
\]
This implies that $\mathbfit{L}_X\mathbfit{1}$ is the zero vector.
Hence $\mathcal{L}_X\left(\sqrt{\mathscr{W}_X}\mathbfit{1}\right)=\sqrt{\mathscr{W}_X}\mathbfit{0}=\mathbfit{0}$.
\end{proof}
An important property of $\mathcal{L}_X$ is that it is the product of the weighted boundary operator matrix and the weighted coboundary operator matrix.
\begin{definition}
Let $K$ be a field, and let $X$ be a $K$-valued vertex-weighted finite graph.
Choose a section $S \to \mathbb{E}(X)$ of the natural map $\mathbb{E}(X) \to E(X)$.
The \emph{weighted boundary operator matrix \glssymbol{bfB}$\mathbfit{B}_X$} of $X$ is a matrix whose rows are labeled by the vertices of $X$, whose columns are labeled by the edges in $S$, and whose $(u,e)$-entry is given by
\begin{equation}
\mathbfit{B}_X(u,e)=
\begin{cases}
\phantom{-}\sqrt{w_{t(e)}}& \text{if $o(e)\neq t(e)$ and $o(e)=u$}, \\
-\sqrt{w_{o(e)}}& \text{if $o(e)\neq t(e)$ and $t(e)=u$}, \\
\hfill 0 \hfill\hfill& \text{otherwise}.
\end{cases}
\end{equation}
$(\mathbfit{B}_X)^{\mathsf{T}}$ is called the \emph{weighted coboundary operator matrix of $X$}.
\end{definition}
\begin{remark}
Let $C_0(X, K)$ denote the group of weighted 0-chains (the free $K$-vector space over $V(X)$), and let $C_1(X, K)$ denote the group of weighted 1-chains (the free $K$-vector space over $S$).
Then $\mathbfit{B}_X$ is a weighted boundary operator from $C_1$ to $C_0$, and $(\mathbfit{B}_X)^{\mathsf{T}}$ is a weighted coboundary operator from $C_0$ to $C_1$.
\end{remark}
\begin{proposition}
Let $X$ be a $K$-valued vertex-weighted finite graph.
Then
\begin{equation}\label{eq_2}
\mathcal{L}_X=\mathbfit{B}_X(\mathbfit{B}_X)^{\mathsf{T}}.
\end{equation}
Note that \eqref{eq_2} does not depend on the choice of sections of $\mathbb{E}(X) \to E(X)$.
\end{proposition}
\begin{proof}
We calculate the product $\mathbfit{B}_X(\mathbfit{B}_X)^{\mathsf{T}}$.
For a diagonal entry,
\begin{align*}
\mathbfit{B}_X(\mathbfit{B}_X)^{\mathsf{T}}(u,u)
&= \sum_{e \in S}\mathbfit{B}_X(u, e)^2 \\
&=\sum_{\substack{e \in S\colon \\ o(e)\neq t(e), \\ o(e)=u}}w_{t(e)} + \sum_{\substack{e \in S\colon \\ o(e)\neq t(e), \\ t(e)=u}}w_{o(e)} \\
&=\sum_{\substack{e \in S\colon \\ o(e)\neq t(e), \\ o(e)=u}}w_{t(e)} + \sum_{\substack{e \in S\colon \\ o(e)\neq t(e), \\ o(\overbar{e})=u}}w_{t(\overbar{e})} \\
&=\sum_{\substack{e \in \mathbb{E}(X)\colon \\ o(e) \neq t(e), \\ o(e)=u}} w_{t(e)}\\
&=\mathbfit{D}_X(u,u)-\mathcal{W}_X(u,u).
\end{align*}
For a non-diagonal entry, 
\begin{align*}
\mathbfit{B}_X(\mathbfit{B}_X)^{\mathsf{T}}(u,v)
&=\sum_{e \in S}\mathbfit{B}_X(u,e)\mathbfit{B}_X(v, e)\\
&=\sum_{\substack{e \in S\colon \\ o(e)=u, \\ t(e)=v}}\left(-\sqrt{w_u}\sqrt{w_v}\right) + \sum_{\substack{e \in S\colon \\ o(\overbar{e})=u, \\ t(\overbar{e})=v}}\left(-\sqrt{w_v}\sqrt{w_u}\right)\\
&= -\sum_{\substack{e \in \mathbb{E}(X)\colon \\ o(e)=u, \\ t(e)=v}}\sqrt{w_u}\sqrt{w_v}\\
&=-\mathcal{W}_X(u,v).
\end{align*}
The above calculations imply that $\mathbfit{B}_X(\mathbfit{B}_X)^{\mathsf{T}}=\mathcal{L}_X$, as desired.
\end{proof}
The equality \eqref{eq_2} allows us to extract the rooted weighted complexity of $X$ from $\mathbfit{B}_X$.
We recall the definitions of weighted complexities.
\begin{definitions}
Let $X$ be a $K$-valued vertex-weighted connected finite graph.
\begin{enumerate}[label=(\roman*)]
\item
For each $v \in V(X)$ and for each spanning tree $T$ of $X$, the \emph{weight \glssymbol{omegaTv}$w(T_v)$ of $T_v$} is the product 
\begin{equation}\label{eq_19}
w(T_v)=\prod_{e \in \mathbb{E}(T_v)}w_{t(e)}.
\end{equation}
\item
For each $v \in V(X)$, the \emph{rooted weighted complexity \glssymbol{kappavX}$\kappa_v(X)$} is the sum 
\begin{equation}
\kappa_v(X)=\sum_{\substack{T\colon \\ \text{$T$ is a spanning tree}}}w(T_v).
\end{equation}
\item The \emph{weighted complexity \glssymbol{kappaX}$\kappa(X)$} is the sum
\begin{equation}
\kappa(X)=\sum_{v \in V(X)}\kappa_v(X).
\end{equation}
\end{enumerate}
\end{definitions}
For each $v \in V(X)$, $\kappa_v(X)$ can be represented as a minor determinant of $\mathbfit{B}_X$.
\begin{proposition}\label{prop_2}
Let $n$ be a positive integer, let $X$ be a $K$-valued vertex-weighted connected finite graph with $n$ vertices, let $v \in V(X)$, and let $A$ be a subset of $S$ consisting of $n-1$ elements.
Consider the square submatrix $\mathbfit{B}_X[\left\{v\right\}, S \setminus A]$ of $\mathbfit{B}_X$ obtained by deleting $v$-th row and the columns not belonging to $A$.
Let \glssymbol{tildeA}$\widetilde{A}$ denote the symmetric directed subgraph of $X$ naturally derived from $A$.
When $A$ is empty, we consider $\widetilde{A}$ as the graph that consists of only $v$.
Then the following hold.
\begin{enumerate}[label=$(\arabic*)$]
\item
If $\widetilde{A}$ is not a tree, then $\det\left(\mathbfit{B}_X[\left\{v\right\},S \setminus A]\right)=0$.
\item
If $\widetilde{A}$ is a tree, then 
$\det\left(\mathbfit{B}_X[\left\{v\right\},S \setminus A]\right)$ is a square root of $w\Bigl(\bigl(\widetilde{A}\bigr)_v\Bigr)$.
\end{enumerate}
\end{proposition}
\begin{proof}[Proof of $(1)$]
If $\widetilde{A}$ is not a tree, then $\widetilde{A}$ contains a cycle graph $C$.
Hence the columns correspond to the edges in $\mathbb{E}(C) \cap S$ are linearly dependent, and hence $\det\left(\mathbfit{B}_X[\left\{v\right\},S \setminus A]\right)=0$.
\end{proof}
\begin{proof}[Proof of $(2)$]
We prove by induction on $n$.
When $n=1$, $\mathbfit{B}_X[\left\{v\right\}, S\setminus A]$ is the 0 by 0 matrix, and hence $\det\left(\mathbfit{B}_X[\left\{v\right\}, S \setminus A]\right)=1$.
On the other hand, $\widetilde{A}$ has no edge, and so $w\Bigl(\bigl(\widetilde{A}\bigr)_v\Bigr)=1$.
This means that our statement is vaild for $n=1$.

Suppose $n>1$.
Let $e$ be an edge in $A$ that has $v$ as its endpoint, and let $u$ be another endpoint of $e$.
Since $\widetilde{A}$ is a tree, the symmetric directed graph obtained by deleting $e$ (and $\overbar{e}$) from $\widetilde{A}$ is the disjoint union of two trees $T_1$ and $T_2$.
We may assume that $T_1$ has $u$, and hence $T_2$ has $v$.
Note that $(\mathbfit{B}_X[\{v\},S\setminus A])[\{u\} , \left\{e\right\}]=\mathbfit{B}_{T_1}[\{u\}, (S \setminus A)\cap \mathbb{E}(T_1)]\oplus \mathbfit{B}_{T_2}[\{v\}, (S \setminus A)\cap \mathbb{E}(T_2)]$.
By the Laplace expansion along the $e$-th column of $\mathbfit{B}_X[\left\{v\right\}, S\setminus A]$ and by the inductive hypothesis, we obtain
\begin{align*}
\det\left(\mathbfit{B}_X[\left\{v\right\}, S\setminus A]\right)
&=(\pm 1)\sqrt{w_u}\det\left(\mathbfit{B}_{T_1}[\{u\}, (S \setminus A)\cap \mathbb{E}(T_1)]\right)\det\left(\mathbfit{B}_{T_2}[\{v\}, (S \setminus A)\cap \mathbb{E}(T_2)]\right) \\
&=(\pm 1)\sqrt{w_u}\sqrt{w((T_1)_u)}\sqrt{w((T_2)_v)} \\
&=(\pm 1)\sqrt{w\left(\left(\widetilde{A}\right)_v\right)}.
\end{align*}
This completes the proof.
\end{proof}
\begin{corollary}[vertex-weighted matrix-tree theorem I]\label{coro_1}
Let $X$ be a $K$-valued vertex-weighted connected finite graph.
Then
\begin{equation}\label{eq_7}
\det\left(\mathcal{L}_X[\left\{v\right\},\left\{v\right\}]\right)=\kappa_v(X)
\end{equation}
holds for every $v \in V(X)$.
\end{corollary}
\begin{proof}
Using the Cauchy-Binet formula, we have
\begin{align*}
\det\left(\mathcal{L}_X[\left\{v\right\},\left\{v\right\}]\right) 
&=\det\left(\left(\mathbfit{B}_X(\mathbfit{B}_X)^{\mathsf{T}}\right)[\left\{v\right\},\left\{v\right\}]\right) \\
&=\sum_{\substack{A\colon \\ A\subseteq S, \\ \#A = n-1}}
\det\left(\mathbfit{B}_X[\{v\}, S \setminus A]\right)\det\left(\mathbfit{B}_X[\{v\}, S \setminus A]\right)^{\mathsf{T}} \\
&=\sum_{\substack{A\colon \\ A\subseteq S, \\ \#A = n-1, \\ \text{$\widetilde{A}$ is a tree}}} \det\left(\mathbfit{B}_X[\{v\}, S \setminus A]\right)^2 & \text{(Use Proposition\ref{prop_2}(1).)}\\
&=\sum_{\substack{T\colon \\ \text{$T$ is a spanning tree}}} w(T_v) & \text{(Use Proposition\ref{prop_2}(2).)}\\
&=\kappa_v(X),
\end{align*}
which completes the proof.
\end{proof}
\begin{remark}
When each vertex has weight 1, Corollary \ref{coro_1} is nothing but the matrix-tree theorem for undirected graphs.
\end{remark}
Corollary \ref{coro_1} states that the principal minors of a Laplacian matrix are the rooted weighted complexity.
More generally, the next theorem reveals the relation between the first minors of a Laplacian matrix and the weighted complexity.
\begin{theorem}[vertex-weighted matrix-tree theorem II]\label{thm_5}
Let $X$ be a $K$-valued vertex-weighted connected finite graph with $n$ vertices.
Then
\begin{equation}\label{eq_8}
\raisebox{-.75ex}{\bigg(}\sum_{z \in V(X)} w_z\raisebox{-.75ex}{\bigg)}\adj (\mathcal{L}_X)(u, v)
=\sqrt{w_u}\sqrt{ w_v}\kappa(X).
\end{equation}
\end{theorem}
\begin{proof}
We may assume that $n\geq 2$ since \eqref{eq_8} is trivial when $n=1$.
We observe that if $\rank(\mathcal{L}_X) \leq n-2$, then $\det\left(\mathcal{L}_X[\left\{u\right\},\left\{v\right\}]\right)=0$ for every $u, v \in V(X)$, and hence $\kappa_v(X)=0$ for every $v \in V(X)$ by \eqref{eq_7}, so \eqref{eq_8} is also trivial in this case.

Suppose $\rank(\mathcal{L}_X)=n-1$.
Then, $\ker(\mathcal{L}_X)$ is one-dimensional and hence is generated by $\sqrt{\mathscr{W}_X}\mathbfit{1}$, which is a non-zero vector.
Let $\adj (\mathcal{L}_X)$ denote the classical adjoint matrix of $\mathcal{L}_X$.
Then $\mathcal{L}_X\cdot \adj (\mathcal{L}_X)=\det\left(\mathcal{L}_X\right)\mathbfit{I}_n =\mathbfit{O}$.
Hence each column of $\adj(\mathcal{L}_X)$ is a scalar multiple of $\sqrt{\mathscr{W}_X}\mathbfit{1}$.
Moreover, each row $\adj(\mathcal{L}_X)$ is also a scalar multiple of $\sqrt{\mathscr{W}_X}\mathbfit{1}$ since $\adj(\mathcal{L}_X)$ is symmetric.
Therefore,
\begin{equation}\label{eq_9}
\adj(\mathcal{L}_X)=\varkappa\cdot \sqrt{\mathscr{W}_X}\mathbfit{J}\sqrt{\mathscr{W}_X}
\end{equation}
for some $\varkappa \in K$. Here $\mathbfit{J}$ denotes the square matrix of order $n$ all of whose entries are $1$.
$\eqref{eq_9}$ implies that 
$\kappa_v(X)=\det\left(\mathcal{L}_X[\left\{v\right\},\left\{v\right\}]\right)=\varkappa\cdot w_v$
for every $v \in V(X)$.
Hence 
\begin{equation}\label{eq_10}
\kappa(X)=\varkappa \cdot \raisebox{-.75ex}{\bigg(}\sum_{v \in V(X)} w_v\raisebox{-.75ex}{\bigg)}.
\end{equation}
Combining \eqref{eq_9} and \eqref{eq_10}, we obtain 
\begin{equation}
\raisebox{-.75ex}{\bigg(}\sum_{v \in V(X)} w_v\raisebox{-.75ex}{\bigg)}\cdot \adj(\mathcal{L}_X)= \kappa(X) \cdot \sqrt{\mathscr{W}_X}\mathbfit{J}\sqrt{\mathscr{W}_X},
\end{equation}
and the proof is complete.
\end{proof}
\begin{remark}
When a symmetric directed graph $X$ has a weight function $w\colon V(X)\to \mathbb{R}_{>0}$, then $\sum_{z \in V(X)}w_z >0$, and \eqref{eq_10} can be expressed as follows.
\begin{equation}\label{eq_11}
\adj(\mathcal{L}_X)=\frac{\kappa(X)}{\sum_{z \in V(X)}w_z }\cdot \sqrt{\mathscr{W}_X}\mathbfit{J}\sqrt{\mathscr{W}_X}.
\end{equation}
\eqref{eq_11} coincides with Theorem 1 in \cite{MR1401006}, so \eqref{eq_8} is a direct generalization of their result.
\end{remark}
\begin{corollary}\label{coro_2}
If $\sum_{z \in V(X)}w_z =0$, then $\kappa(X)=0$.
\end{corollary}
\begin{proof}
When all weights are zero, $\kappa(X)=0$ by \eqref{eq_19}.
If there is a vertex $v$ whose weight is non-zero, then, by \eqref{eq_8},
\[
\kappa(X) = \frac{1}{w_v} \raisebox{-.75ex}{\bigg(}\sum_{z \in V(X)}w_z\raisebox{-.75ex}{\bigg)}\det\left(\mathcal{L}_X[\{v\}, \{v\}]\right)=0,
\]
and our statement follows.
\end{proof}
\subsection{Voltage assignments}\label{subsec_2.2}
In this subsection, we deal with the Galois theory for vertex-weighted graphs.
We first recall the concept of covers of graphs.
\begin{definitions}
Let $X$ and $Y$ be $K$-valued vertex-weighted graphs.
\begin{enumerate}[label=(\roman*)]
\item
A \emph{morphism $f$ from $Y$ to $X$} is a pair $\left(f_V, f_{\mathbb{E}}\right)$, where $f_V$ is a map from $V(Y)$ to $V(X)$ and $f_{\mathbb{E}}$ is a map from $\mathbb{E}(Y)$ to $\mathbb{E}(X)$ satisfying the following three conditions.
\begin{enumerate}[label=$(\alph*)$]
\item
$f_V(o(e))=o(f_{\mathbb{E}}(e))$, and $f_V(t(e))=t(f_{\mathbb{E}}(e))$ for every $e\in \mathbb{E}(Y)$.
\item $(f_{\mathbb{E}}(e))\overbar{\hphantom{I}}=f_{\mathbb{E}}(\overbar{e})$ for every $e \in \mathbb{E}(Y)$.
\item $w_{f_V(v)}=w_v$ for every $v \in V(Y)$.
\end{enumerate}
\item Under the additonal assumption $X$ and $Y$ are both connected, a (\emph{unramified}) \emph{cover} is a morphism $\pi\colon Y \to X$ satisfying the following two conditions.
\begin{enumerate}[label=$(\alph*)$]\setcounter{enumii}{3}
\item
$\pi_V$ and $\pi_{\mathbb{E}}$ are surjective.
\item For each $v \in V(Y)$, the restriction 
\[
\pi_{\mathbb{E}}|_{\mathbb{E}(Y)_v} \colon \mathbb{E}(Y)_v \to \mathbb{E}(X)_{\pi_V(v)}
\]
is bijective, where $\mathbb{E}(Y)_v=\{\,e \in \mathbb{E}(Y)\;|\; o(e)=v\,\}$.
\end{enumerate}
We write $Y/X$ for a cover $\pi\colon Y \to X$ as usual.
\item 
For a cover $Y/X = \pi$, an \emph{automorphism of $Y/X$} is an automorphism $f\colon Y\similarrightarrow Y$ such that $\pi \circ f = \pi$.
We write $\Aut(Y/X)$ for the set of all automorphisms of $Y/X$.
\item 
A cover $Y/X=\pi$ is said to be \emph{Galois} if $\#\Aut(Y/X)=\#\pi_V^{-1}(v)$ for each $v \in V(X)$. For a Galois cover $Y/X$, $\Gal(Y/X)$ stands for $\Aut(Y/X)$.
\end{enumerate}
\end{definitions}

\begin{definition}
\begin{enumerate}[label=(\roman*)]
Let $X$ be a $K$-valued vertex-weighted graph.
\item
For a group $G$, a map $\alpha\colon \mathbb{E}(X)\to G$ is called a \emph{voltage assignment} if $\alpha(\overbar{e})=\alpha(e)^{-1}$ for each $e \in \mathbb{E}(X)$.
\item
For a voltage assignment $\alpha \colon \mathbb{E}(X) \to G$,
set $V(X(\alpha))=V(X) \times G$, set $\mathbb{E}(X(\alpha))=\mathbb{E}(X) \times G$, and let the map $\inc_{X(\alpha)}\colon \mathbb{E}(X(\alpha))\to V(X(\alpha)) \times V(X(\alpha))$ be given  by setting $o((e, \sigma))=(o(e), \sigma)$, $t((e, \sigma))=(o(e), \sigma\alpha(e))$ for each $(e, \sigma) \in \mathbb{E}(X(\alpha))$.
Then the triple $(X, G)$$=$$\bigl(V(X(\alpha)),$ $\mathbb{E}(X(\alpha)),$ $\inc_{X(\alpha)}\bigr)$ is a symmetric directed graph with the map $(e, \sigma)\mapsto (e, \sigma)^{\overbar{}}=(\overbar{e}, \sigma\alpha(e))$.
Moreover, $X(\alpha)$  has the natural vertex-weighted structure given by $w_{(v,\sigma)} = w_v$.
The $K$-valued vertex-weighted graph \glssymbol{Xalpha}$X(\alpha)$ is called the \emph{derived graph of $\alpha$}.
\end{enumerate}
\end{definition}
Due to the following facts according to \cite{Gon21}, we would rather consider voltage assignments and derived graphs than general covers.
The natural cover $\pi\colon X(G,{\alpha})\to X$ defined by $\pi_V(v,\sigma)=v$ and $\pi_{\mathbb{E}}(e,\sigma)=e$ is a Galois cover whose automorphism group is isomorphic to $G$ if $X(\alpha)$ is connected.
Conversely, every finite Galois cover $Y/X$ is isomorphic to a cover induced by a voltage assignment whose target is $\Gal(Y/X)$. 
The above facts are also valid for $K$-valued vertex-weighted graphs.
In particular, \cite[Theorem 4]{Gon21} gives us a characterization of the condition in which $X(G,{\alpha})$ is connected.
\begin{example}\label{ex_4}
Consider the vertex-weighted graph $X$ and voltage assigmnment $\alpha\colon \mathbb{E}(X) \to \mathbb{Z}/2\mathbb{Z}$ in Figure \ref{fig_100}.
\begin{figure}
\centering
\begin{minipage}{8em}
\centering
\begin{tikzpicture}[scale=1]
\coordinate[label=below:$v_1$,draw, circle, inner sep=2pt, fill=black](v1)at(0,0);
\coordinate[label=above left:$v_2$,draw, circle, inner sep=2pt, fill=black](v2)at({2*cos(pi/3 r)}, {2*sin(pi/3 r)});
\coordinate[label=below:$v_3$,draw, circle, inner sep=2pt, fill=black](v3)at(2,0);

\draw(v1)[->, very thick] to  node[above left]{$e_1$} (v2);
\draw(v3)[->, very thick] to  node[below]{$e_3$} (v1);
\draw(v2)[->, very thick]  to node[right]{$e_2$} (v3);
\draw(v1)[->, very thick]to[out=210, in=510, loop] node[left]{$e_4$} ();
\end{tikzpicture}
\end{minipage}
\hspace{3em}
\begin{minipage}{12em}
The weights of $X$ is given by
\begin{flushleft}
$
\begin{cases}
w_{v_1}=\sqrt{2},\\
w_{v_2}=1,\\
w_{v_3}=1.
\end{cases}
$
\end{flushleft}
\phantom{$($}
\end{minipage}
\begin{minipage}{0.03\columnwidth}
\phantom{a}
\end{minipage}
\begin{minipage}{11em}
$\alpha \colon \mathbb{E}(X)\to \mathbb{Z}_2$ is given by
\begin{flushleft}
$
\begin{cases}
\alpha(e_1)=1, \\\alpha(e_2)=0, \\
\alpha(e_3)=0, \\\alpha(e_4))=1.
\end{cases}
$
\end{flushleft}
\end{minipage}
\caption{The vertex-weighted graph $X$ and the voltage assignment $\alpha$ in Example \ref{ex_4}}\label{fig_100}
\end{figure}
Then Figure \ref{fig_4} illustlates $X(\alpha)$.
Under the natural cover morphism, 
vertices $u_i$ and $u'_i$ map to $v_i$ for each $i \in \{1,2,3\}$, and edges $f_i$ and $f'_i$ map to $e_i$ for each $i \in \{1,2,3,4\}$.
\begin{figure}
\begin{tikzpicture}[scale=1.1]
\coordinate[label=below left:${u_1}$,draw, circle, inner sep=1pt, fill=black](v10)at({cos(7*pi/6 r)}, {sin(7*pi/6 r)});
\coordinate[label=above:${u_2}$,draw, circle, inner sep=1pt, fill=black](v20)at({cos(pi/2 r)}, {sin(pi/2 r)});
\coordinate[label=below right:${u_3}$,draw, circle, inner sep=1pt, fill=black](v30)at({cos(11*pi/6 r)}, {sin(11*pi/6 r)});
\coordinate[label=below left:$w'_1$,draw, circle, inner sep=1pt, fill=black](v11)at({2*cos(7*pi/6 r)}, {2*sin(7*pi/6 r)});
\coordinate[label=above:$w'_2$,draw, circle, inner sep=1pt, fill=black](v21)at({2*cos(pi/2 r)}, {2*sin(pi/2 r)});
\coordinate[label=below right:$w'_3$,draw, circle, inner sep=1pt, fill=black](v31)at({2*cos(11*pi/6 r)}, {2*sin(11*pi/6 r)});

\draw(v10)[->, thick] to  node[above left]{$f_1$} (v21);
\draw(v30)[->, thick] to  node[below]{$f_3$} (v10);
\draw(v20)[->, thick]  to node[right]{$f_2$} (v30);
\draw(v11)[->, thick] to  node[above left]{$f'_1$} (v20);
\draw(v31)[->, thick] to  node[below]{$f'_3$} (v11);
\draw(v21)[->, thick]  to node[right]{$f'_2$} (v31);
\draw(v10)[->, thick]to[bend left=90] node[below]{$f_4$} (v11);
\draw(v10)[->, thick]to[bend right=90] node[above left]{$f'_4$} (v11);
\end{tikzpicture}
\caption{the derived graph $X(\alpha)$ in in Example \ref{ex_4}}\label{fig_4}
\end{figure}
\end{example}
A first glimpse of relations between $X(\alpha)$ and $X$ is a relation of the weighted matrices.
\begin{definition}
Let $X$ be a $K$-valued vertex-weighted finite graph, let $G$ be a finite group, and let $\alpha\colon \mathbb{E}(X) \to G$ be a voltage assignment.
Then, for $\sigma \in G$, we write \glssymbol{calWXalphasigma}$\mathcal{W}_X^{\alpha, \sigma}$ for the square matrix labeled by the vertices of $X$ and whose $(u,v)$-entry is given by 
\begin{equation}
\mathcal{W}_X^{\alpha, \sigma}(u,v)=\sum_{\substack{e \in \mathbb{E}(X)\colon \alpha(e)=\sigma, \\ o(e)=u,\ t(e)=v}} 
\sqrt{w_u}\sqrt{w_v}.
\end{equation}
Note that $\mathcal{W}_X=\sum_{\sigma \in G}\mathcal{W}_X^{\alpha, \sigma}$.
\end{definition}
\begin{proposition}\label{prop_3}
Let $X$ be a $K$-valued vertex-weighted finite graph, let $G$ be a finite group, and let $\alpha\colon \mathbb{E}(X) \to G$ be a voltage assignment.
Then 
\begin{equation}\label{eq_12}
\mathbfit{D}_{X(\alpha)}=\mathbfit{D}_X\otimes \mathbfit{I}_G,
\end{equation}
where $\otimes$ stands for the Kronecker products for matrices and where $\mathbfit{I}_G$ is the identity matrix labeled by $G$.
Moreover, 
we consider the right regular representation $\sigma$ of $G$ as a map $G \to \mathrm{Mat}_{G \times G}(K)$, where $\mathrm{Mat}_{G \times G}(K)$ stands for the set of square matrices over $K$ labeled by $G$.
Then 
\begin{equation}\label{eq_13}
\mathcal{W}_Y=\sum_{\sigma  \in G}\mathcal{W}_X^{\alpha, \sigma}\otimes \rho(\sigma).
\end{equation}
\end{proposition}
\begin{proof}[Proof of \eqref{eq_12}]
Recall that $\mathbfit{D}_X\otimes \mathbfit{I}_G$ is labeled by $V(X)\times G$, and the total order on $V(X) \times G$ is given by the lexicographic order induced by total orders on $V(X)$ and $G$.
The $((u, \tau), (v, \zeta))$-entry of $\mathbfit{D}_X\otimes \mathbfit{I}_G$ is 
\[
\mathbfit{D}_X(u,v)\cdot \mathbfit{I}_G(\tau, \zeta)=
\begin{cases}
\sum_{\substack{e \in \mathbb{E}(X)\colon \\ o(e)=u}}w_{t(e)} & \text{if $u=v$ and $\tau=\zeta$}, \\
\hfill 0 \hfill\hfill &\text{otherwise.}
\end{cases}
\]
Hence $\mathbfit{D}_X\otimes \mathbfit{I}_G$ is a diagonal matrix and $((u,u),(\tau,\tau))$-entry is
\[
\sum_{\substack{e \in \mathbb{E}(X)\colon \\ o(e)=u}}w_{t(e)}=\sum_{\substack{(e, \tau) \in \mathbb{E}(X(\alpha))\colon \\ o(e, \tau)=(u,\tau)}}w_{t(e,\tau)},
\]
which implies that $\mathbfit{D}_X\otimes \mathbfit{I}_G=\mathbfit{D}_{X(\alpha)}$.
\end{proof}
\begin{proof}[Proof of \eqref{eq_13}]
The $((u, \tau), (v, \zeta))$-entry of $\mathcal{W}_X^{\alpha, \sigma}\otimes \rho(\sigma)$ is 
\[
\mathcal{W}_X^{\alpha, \sigma}(u,v)\cdot \rho(\sigma)(\tau, \zeta)=
\begin{cases}
\sum_{\substack{e \in \mathbb{E}(X)\colon \alpha(e)=\sigma, \\ o(e)=u,\ t(e)=v}}\sqrt{w_u}\sqrt{w_v}  & \text{if $\zeta=\tau\sigma$,} \\
\hfill 0 \hfill\hfill &\text{otherwise.}
\end{cases}
\]
Hence $((u, \tau), (v, \zeta))$-entry of $\sum_{\sigma \in G}\mathcal{W}_X^{\alpha, \sigma}\otimes \rho (\sigma)$ is 
\[
\sum_{\substack{e \in \mathbb{E}(X)\colon \zeta=\tau\alpha(e) \\ o(e)=u,\ t(e)=v}} \sqrt{w_u}\sqrt{w_v} 
=\sum_{\substack{(e, \tau)\in \mathbb{E}(X(\alpha))\colon \\ o(e,\tau)=(u, \tau), \\ t(e,\tau)=(v, \zeta)}}
\sqrt{w_u}\sqrt{w_v}
=\mathcal{W}_Y((u,\tau), (v,\zeta)),
\]
which completes the proof.
\end{proof}
\section{$h$-functions}\label{sec_3}
This section forms the technical heart of this paper.
Our aim in this subsection is to describe a relation of weighted complexities for a voltage cover.
\begin{definition}
Let $X$ be a $K$-valued vertex-weighted finite graph, let $G$ be a finite group, let $\alpha\colon \mathbb{E}(X)\to G$ be a voltage assignment, and let $\rho$ be a representation of $G$ over $K$ of degree $d_\rho$.
We define \glssymbol{hXalphapsit}$h_X^{\alpha}(\rho, t)$ to be a polynomial in $t$ given by
\begin{equation}
h_X^{\alpha}(\rho,t)=\det\left(
\mathbfit{I}_{V(X)}\otimes \mathbfit{I}_{d_{\rho}}-t\cdot\sum_{\sigma \in G}\mathcal{W}_X^{\alpha, \sigma}\otimes \rho(\sigma) +t^2\cdot \left(\mathbfit{D}_X-\mathbfit{I}_{V(X)}\right)\otimes \mathbfit{I}_{d_{\rho}}
\right).
\end{equation}
\end{definition}
\begin{remark}
In general, there exists a unique permutation square matrix $\mathbfit{R}$ of order $mn$ satisfying the following property:
$\mathbfit{R}^{\mathsf{T}}(\mathbfit{A} \otimes \mathbfit{B})\mathbfit{R}=\mathbfit{B} \otimes \mathbfit{A}$
for every square matrix $\mathbfit{A}$ of order $m$ and for every square matrix $\mathbfit{B}$ of order $n$.
Since permutation matrix is orthogonal, we deduce that $\det(\mathbfit{A}\otimes \mathbfit{B})=\det(\mathbfit{B}\otimes \mathbfit{A})$.
Hence we obtain
\begin{equation}\label{eq_16}
h_X^{\alpha}(\rho,t)=\det\left(
\mathbfit{I}_{d_{\rho}}\otimes\mathbfit{I}_{V(X)}
-t\cdot\sum_{\sigma \in G}\rho(\sigma)\otimes \mathcal{W}_X^{\alpha, \sigma}
+t^2\cdot \mathbfit{I}_{d_{\rho}}\otimes \left(\mathbfit{D}_X-\mathbfit{I}_{V(X)}\right)
\right).
\end{equation}
\end{remark}

Now, we prove one of our main theorems, which is the so-called decomposition formula. In \cite{WFS11}, Wu-Feng-Sato proves the decomposition formula for vertex-weighted complexities. We deduce their formula by proving one for rooted vertex-weighted complexities. Also, we remove the assumption $\Big(\sum_{e\in\mathbb{E}(X)}w_{o(e)}w_{t(e)}\Big)/2- \sum_{v\in V(X)}w_v\neq 0$, which Wu-Feng-Sato tacitly assumes.
\begin{theorem}\label{thm_6}
Let $X$ be a $K$-valued vertex-weighted connected finite graph, let $G$ be a finite group with $\operatorname{char} K \not\mid \# G$, let $\alpha\colon \mathbb{E}(X)\to G$ be a voltage assignment, and let $\widehat{G}$ be a full set of representatives of equivalence classes for the irreducible representations of $G$.
Then, for every $v \in V(X)$ and for every $\sigma \in G$,
\begin{equation}\label{eq_14}
\kappa_{(v, \sigma)}(X(\alpha))=\frac{\kappa_v(X)}{\#G}\cdot\prod_{\rho \in \widehat{G} \setminus \{\mathrm{triv}_G\}}
\left(h_X^{\alpha}(\rho,1)\right)^{d_{\rho}},
\end{equation}
where $\mathrm{triv}_G$ stands for the trivial representation of $G$.
Hence, summing up \eqref{eq_14} over all vertices of $X(\alpha)$, 
\begin{equation}\label{eq_21}
\kappa(X(\alpha))=\kappa(X)\cdot\prod_{\rho \in \widehat{G} \setminus \{\mathrm{triv}_G\}}
\left(h_X^{\alpha}(\rho,1)\right)^{d_{\rho}}
\end{equation}
holds.
\end{theorem}
\begin{proof}
We write $\reg_G$ for the right regular representation of $G$.
Then, since $\operatorname{char} K \not\mid \# G$,  there is an invertible matrix $\mathbfit{P}$ such that $\mathbfit{P}^{-1}\reg_G(\sigma)\mathbfit{P}=\mathrm{triv}_G(\sigma)\oplus
\left(\bigoplus_{\rho\in \widehat{G}\setminus \{\mathrm{triv}_G\}}\rho(\sigma)^{\oplus d_{\rho}}\right)$.
Since $\reg_G(\sigma) \mathbfit{1} =\mathbfit{1}$ for the vector $\mathbfit{1}$ whose entries are all 1,
we may assume that $\sigma$-th column of $\mathbfit{P}$ is $\mathbfit{1}$.
Hence the sum of all entries of the $\sigma$-th row of $\mathbfit{P}^{-1}$ is 1.

Assume that $v$ is the smallest element in $V(X)$, and $\sigma$ is the smallest element in $G$.
Let $\mathbfit{R}_{V(X),G}$ denote the permutation matrix that commutes all Kronecker products of square matrices labeled by $V(X)$ and square matrices labeled by $G$.
Then the $(v,\sigma)$-th column of $\mathbfit{R}_{V(X),G}$ is the $(v, \sigma)$-th vector of the standard basis.
By \eqref{eq_12} and \eqref{eq_13}, we have
\begin{align}
&(
\mathbfit{I}_{V(X)}\otimes \mathbfit{P}
)^{-1}
\mathcal{L}_{X(\alpha)}
(
\mathbfit{I}_{V(X)}\otimes \mathbfit{P}
) \\
&=\left(
\mathbfit{I}_{V(X)}\otimes \mathbfit{P}
\right)^{-1}
\left(
\mathbfit{D}_X\otimes\mathbfit{I}_G - \sum_{\sigma \in G}\mathcal{W}_X^{\alpha, \sigma}\otimes \reg_G(\sigma)
\right)
\left(
\mathbfit{I}_{V(X)}\otimes \mathbfit{P}
\right) \\
&=
\mathbfit{D}_X\otimes\mathbfit{I}_G - \sum_{\sigma \in G}\mathcal{W}_X^{\alpha, \sigma}\otimes 
\left(
\mathrm{triv}_G(\sigma)\oplus
\left(\bigoplus_{\rho\in \widehat{G}\setminus \{\mathrm{triv}_G\}}\rho(\sigma)^{\oplus d_{\rho}}\right)
\right) \\
&=(\mathbfit{R}_{V(X),G})^{\mathsf{T}}\left(
\mathbfit{I}_G\otimes \mathbfit{D}_X - \sum_{\sigma \in G}
\left(
\left(
\mathrm{triv}_G(\sigma)\oplus
\left(\bigoplus_{\rho\in \widehat{G}\setminus \{\mathrm{triv}_G\}}\rho(\sigma)^{\oplus d_{\rho}}\right)
\right)
\otimes \mathcal{W}_X^{\alpha, \sigma}
\right)
\right)\mathbfit{R}_{V(X),G}\\
&=(\mathbfit{R}_{V(X),G})^{\mathsf{T}}\left(
\mathbfit{I}_G\otimes \mathbfit{D}_X - \sum_{\sigma \in G}
\left(
\mathcal{W}_X^{\alpha, \sigma}\oplus
\left(\bigoplus_{\rho\in \widehat{G}\setminus \{\mathrm{triv}_G\}}\rho(\sigma)^{\oplus d_{\rho}}\otimes \mathcal{W}_X^{\alpha, \sigma}\right)
\right)
\right)\mathbfit{R}_{V(X),G}\\
&=(\mathbfit{R}_{V(X),G})^{\mathsf{T}}\left(
\left(
\mathbfit{D}_X-\sum_{\sigma \in G}\mathcal{W}_X^{\alpha, \sigma}
\right)
\oplus
\left(\bigoplus_{\rho\in \widehat{G}\setminus \{\mathrm{triv}_G\}}\hspace{-3pt}
\left(
(\mathbfit{I}_{d_{\rho}})^{\oplus d_{\rho}}\otimes\mathbfit{D}_X
-\sum_{\sigma \in G}
\rho(\sigma)^{\oplus d_{\rho}}\otimes \mathcal{W}_X^{\alpha, \sigma}\right)
\right)
\right)\mathbfit{R}_{V(X),G} \\
&=(\mathbfit{R}_{V(X),G})^{\mathsf{T}}\left(
\mathcal{L}_X
\oplus
\left(\bigoplus_{\rho\in \widehat{G}\setminus \{\mathrm{triv}_G\}}
\left(
\mathbfit{I}_{d_{\rho}}\otimes\mathbfit{D}_X
-\sum_{\sigma \in G}
\rho(\sigma)\otimes \mathcal{W}_X^{\alpha, \sigma}\right)^{\oplus d_{\rho}}
\right)
\right)\mathbfit{R}_{V(X),G} \\
&=(\mathbfit{R}_{V(X),G})^{\mathsf{T}}\left(
\mathcal{L}_X
\oplus
\left(\bigoplus_{\rho\in \widehat{G}\setminus \{\mathrm{triv}_G\}}
\left(\mathbfit{H}_X^{\alpha}(\rho, 1)\right)^{\oplus d_{\rho}}
\right)
\right)\mathbfit{R}_{V(X),G}, \label{eq_15}
\end{align}
where \glssymbol{bmHXalphapsit}$\mathbfit{H}_X^{\alpha}(\rho, t)=\mathbfit{I}_{d_{\rho}}\otimes\mathbfit{I}_{V(X)}
-t\cdot\sum_{\sigma \in G}\rho(\sigma)\otimes \mathcal{W}_X^{\alpha, \sigma}
+t^2\cdot \mathbfit{I}_{d_{\rho}}\otimes \left(\mathbfit{D}_X-\mathbfit{I}_{V(X)}\right)$, which is the matrix appearing in \eqref{eq_16}.
Hence, applying the Cauchy-Binet formula to \eqref{eq_15},
\begin{align*}
\det\bigl(&(\mathbfit{I}_{V(X)}\otimes \mathbfit{P}
)^{-1}
\mathcal{L}_{X(\alpha)}
(
\mathbfit{I}_{V(X)}\otimes \mathbfit{P}))[\{(v,\sigma)\},\{(v, \sigma)\}]\bigr) \\
&=\sum_{(u, \tau),\ (x,\zeta) \in V(X)\times G}
\left(
\begin{gathered}
\det\bigl((\mathbfit{R}_{V(X),G})^{\mathsf{T}}[\{(v,\sigma)\},\{(u, \tau)\}]\bigr)\cdot \\
\det\left(\left(\mathcal{L}_X\oplus\left(\bigoplus_{\rho\in \widehat{G}\setminus \{\mathrm{triv}_G\}}\left(\mathbfit{H}_X^{\alpha}(\rho, 1)\right)^{\oplus d_{\rho}}\right)
\right)[\{(u,\tau)\},\{(x,\zeta)\}]\right) \cdot \\
\det\bigl(\mathbfit{R}_{V(X),G}[\{(x,\zeta)\},\{(v, \sigma)\}]\bigr)
 \end{gathered} 
 \right) \\
&=\det\left(\left(\mathcal{L}_X\oplus\left(\bigoplus_{\rho\in \widehat{G}\setminus \{\mathrm{triv}_G\}}\left(\mathbfit{H}_X^{\alpha}(\rho, 1)\right)^{\oplus d_{\rho}}\right)
\right)[\{(v,\sigma)\},\{(v, \sigma)\}]\right) \\
&=\det\left( \mathcal{L}_X[\{v\},\{v\}]\right) \prod_{\rho\in\widehat{G} \setminus \{\mathrm{triv}_G\}}\det\left(\mathbfit{H}_X^{\alpha}(\rho, 1)\right)^{d_{\rho}} \quad \text{(Note that $(v, \sigma)$ is the smallest in $V(X)\times G$.)}\\
&=\kappa_v(X)\prod_{\rho\in\widehat{G} \setminus \{\mathrm{triv}_G\}}\left(h_X^{\alpha}(\rho, 1)\right)^{d_{\rho}}. \numberthis \label{eq_17}
\end{align*}
On the other hand, pick $\varkappa \in \overbar{\mathbb{Q}}_p$ that satisfies $\adj\left(\mathcal{L}_{X(\alpha)}\right)=\varkappa\cdot \sqrt{\mathscr{W}_{X(\alpha)}}\ \mathbfit{J}\sqrt{\mathscr{W}_{X(\alpha)}}$ as in \eqref{eq_9}.
Then, by using the Cauchy-Binet formula again, we have
\begin{align*}
\det&((\mathbfit{I}_{V(X)}\otimes \mathbfit{P}
)^{-1}
\mathcal{L}_{X(\alpha)}
(
\mathbfit{I}_{V(X)}\otimes \mathbfit{P}))[\{(v,\sigma)\},\{(v, \sigma)\}] \\
&=\sum_{(u, \tau),\ (x,\zeta) \in V(X)\times G}
\left(
\begin{gathered}
\det\biggl(\left(\mathbfit{I}_{V(X)}\otimes \mathbfit{P}\right)^{-1}[\{(v,\sigma)\},\{(u, \tau)\}]\biggr)\cdot \\
\det\left(\left(\mathcal{L}_{X(\alpha)}
\right)[\{(u,\tau)\},\{(x,\zeta)\}]\right) \cdot \\
\det\bigl(\left(\mathbfit{I}_{V(X)}\otimes \mathbfit{P}\right)[\{(x,\zeta)\},\{(v, \sigma)\}]\bigr)
 \end{gathered} 
 \right)  \\
 &=\sum_{(u, \tau),\ (x,\zeta) \in V(X)\times G}
\left(
\begin{gathered}
\adj\left(\frac{1}{\det\left(\mathbfit{I}_{V(X)}\otimes\mathbfit{P}\right)}\adj\left(\mathbfit{I}_{V(X)}\otimes \mathbfit{P}\right)\right)((u,\tau),(v, \sigma))\cdot \\
\varkappa\sqrt{w_{(u, \tau)}}\sqrt{w_{(x,\zeta)}}\ \cdot \\
\adj\left(\mathbfit{I}_{V(X)}\otimes \mathbfit{P}\right)((v, \sigma), (x,\zeta))
 \end{gathered} 
 \right) \\
&\overset{\heartsuit}{=}\sum_{(u, \tau),\ (x,\zeta) \in V(X)\times G}
\left(
\begin{gathered}
\left(\frac{1}{\det\left(\mathbfit{I}_{V(X)}\otimes\mathbfit{P}\right)}\left(\mathbfit{I}_{V(X)}\otimes \mathbfit{P}\right)\right)((u,\tau),(v, \sigma))\cdot \\
\varkappa\sqrt{w_u}\sqrt{w_x}\ \cdot \\
\det\left(\mathbfit{I}_{V(X)}\otimes\mathbfit{P}\right)\cdot\left(\mathbfit{I}_{V(X)}\otimes \mathbfit{P}^{-1}\right)((v, \sigma),(x,\zeta))
 \end{gathered} 
 \right) \\
 &=\sum_{(u, \tau),\ (x,\zeta) \in V(X)\times G}
\left(
\begin{gathered}
\left(\mathbfit{I}_{V(X)}\otimes \mathbfit{P}\right)((u,\tau),(v, \sigma))\cdot \\
\varkappa\sqrt{w_u}\sqrt{w_x}\ \cdot \\
\left(\mathbfit{I}_{V(X)}\otimes \mathbfit{P}^{-1}\right)((v, \sigma),(x,\zeta))
 \end{gathered} 
 \right) \\
 &=\sum_{\tau,\zeta \in G}\varkappa w_v \cdot \mathbfit{P}^{-1}(\sigma, \zeta)\quad\text{(Recall that $\sigma$-th column of $\mathbfit{P}$ is $\mathbfit{1}$.)}\\ 
 &=\#G\cdot \varkappa w_v\quad\text{(Recall that the sum of all entries of $\sigma$-th row of $\mathbfit{P}^{-1}$ is 1.)} \\
 &= \#G\cdot \kappa_{(v,\sigma)}(X(\alpha)).\numberthis \label{eq_18}
\end{align*}
In $(\heartsuit)$, we use the fact that $\adj\left(\frac{1}{\det(A)}\adj(A)\right)=\frac{1}{\det(A)}A$ for an invertible matrix $A$.
Therefore, combining \eqref{eq_17} with \eqref{eq_18}, we obtain \eqref{eq_14}.
\end{proof}
\begin{remark}
The manner of our proof is still valid for \cite[Thm. 3.2]{AMT},
and hence we can generalize their statement to the case that $G$ is non-abelian.
\end{remark}
\subsection{Direct sum rule and induction property for $h$-functions}\label{subsec_3.1}
In this subsection, we establish the direct sum rule and the induction property for $h$-functions.

The following is the direct sum rule.
\begin{theorem}
Let $X$ be a $K$-valued vertex-weighted finite graph, let $\alpha\colon \mathbb{E}(X)\to G$ be a voltage assignment with a finite group $G$.
Then
\begin{equation}\label{eq_26}
h_X^{\alpha}(\rho_1\oplus \rho_2, t)=h_X^{\alpha}(\rho_1, t)\cdot h_X^{\alpha}(\rho_2, t)
\end{equation}
for every representations $\rho_1$ and $\rho_2$ of $G$.
\end{theorem}
\begin{proof}
Using the expression \eqref{eq_16}, we have
\begin{align*}
h_X^{\alpha}(\rho_1\oplus \rho_2. t) &=\det\left(
\mathbfit{I}_{d_{\rho_1\oplus \rho_2}}\otimes\mathbfit{I}_{V(X)}
-t\cdot \sum_{\sigma \in G}\bigl((\rho_1\oplus\rho_2)(\sigma)\bigr)\otimes\mathcal{W}_X^{\alpha, \sigma}
+t^2\cdot \mathbfit{I}_{d_{\rho_1\oplus \rho_2}}\otimes \left(\mathbfit{D}_X-\mathbfit{I}_{V(X)}\right)
\right)\\
&=\det\left(
\begin{gathered}
\left(\mathbfit{I}_{d_{\rho_1}}\otimes\mathbfit{I}_{V(X)}\right)\oplus \left(\mathbfit{I}_{d_{\rho_2}}\otimes\mathbfit{I}_{V(X)}\right)\\
-t\cdot \sum_{\sigma \in G}\bigl(\rho_1(\sigma)\otimes\mathcal{W}_X^{\alpha, \sigma}\bigr)\oplus \bigl(\rho_2(\sigma)\otimes\mathcal{W}_X^{\alpha, \sigma}\bigr)\\
+t^2\cdot \left(\mathbfit{I}_{d_{\rho_1}}\otimes(\mathbfit{D}_X-\mathbfit{I}_{V(X)})\right)\oplus \left(\mathbfit{I}_{d_{\rho_2}}\otimes(\mathbfit{D}_X-\mathbfit{I}_{V(X)})\right)
\end{gathered}
\right) \\
&=\det\left(
\begin{gathered}
\left(\mathbfit{I}_{d_{\rho_1}}\otimes\mathbfit{I}_{V(X)}
-t\cdot\sum_{\sigma \in G}\rho(\sigma)\otimes \mathcal{W}_X^{\alpha, \sigma}
+t^2\cdot \mathbfit{I}_{d_{\rho_1}}\otimes \left(\mathbfit{D}_X-\mathbfit{I}_{V(X)}\right)
\right) \\
\oplus\left(\mathbfit{I}_{d_{\rho_2}}\otimes\mathbfit{I}_{V(X)}
-t\cdot\sum_{\sigma \in G}\rho_2(\sigma)\otimes \mathcal{W}_X^{\alpha, \sigma}
+t^2\cdot \mathbfit{I}_{d_{\rho_2}}\otimes \left(\mathbfit{D}_X-\mathbfit{I}_{V(X)}\right)
\right) 
\end{gathered}
\right) \\
&=h_X^{\alpha}(\rho_1, t) \cdot h_X^{\alpha}(\rho_2, t),
\end{align*}
as desired.
\end{proof}

Next, we state the induction property for $h$-functions.
We first recall some facts of intermediate graphs of a Galois cover $X(\alpha)/X$.
Let $X$ be a $K$-valued vertex-weighted finite graph, let $\alpha\colon \mathbb{E}(X)\to G$ be a voltage assignment with a finite group $G$, and let $H$ be a subgroup of $G$.
An intermediate graph $Z$ corresponding to $H$ may be  given by 
\begin{equation}\label{eq_22}
\begin{cases}
V(Z)=V(X)\times (H \backslash G), \\
\mathbb{E}(Z)=\mathbb{E}(X) \times  (H \backslash G), \\
\inc_Z\colon \left(e, H\sigma\right) \mapsto \bigl(\left(o(e), H\sigma\right), \left(t(e), H\sigma \alpha(e)\right)\bigr).
\end{cases}
\end{equation}
If we choose a full set $\{\sigma_1,\dotsc,\sigma_m\}$ of representatives of $H\backslash G$, then we obtain a voltage assignment $\beta\colon \mathbb{E}(Z) \to H$ as follows.
For $(e, H \sigma_i) \in \mathbb{E}(Z)$, there exists a unique $j \in \{1,\dotsc, m\}$ such that $H\sigma_i\alpha(e)=H\sigma_j$.
Now let $\beta\bigl((e, H \sigma_i)\bigr)=\sigma_i\alpha(e){\sigma_j}^{-1}$.
We may directly check that $\beta$ is a voltage assignment, and $Z(\beta) \cong X(\alpha)$ under the identifications $\bigl((v, H\sigma_i),\tau\bigr) \in V(Z(\beta))$ with $(v, \tau \sigma_i) \in V(X(\alpha))$ and $\bigl((e, H\sigma_i),\tau\bigr) \in \mathbb{E}(Z(\beta))$ with $(e, \tau \sigma_i) \in \mathbb{E}(X(\alpha))$.
In the remainder of this subsection, we fix a full set of representatives of $H\backslash G$ and induced voltage assignment $\beta$.
\begin{lemma}
Let $\rho\colon H \to \GL(V)$ be a representation of $H$ over $K$, and let $\Ind_H^G\rho$ stand for the induced representation of $\rho$.
Then
\begin{equation}\label{eq_23}
\sum_{\xi \in G}\left( \mathcal{W}_X^{\alpha, \xi}\otimes \left(\Ind_H^G\rho\right)(\xi)\right)=\sum_{\tau \in H}\left(\mathcal{W}_{Z}^{\beta,\tau}\otimes \rho(\tau)\right).
\end{equation}
\end{lemma}
\begin{proof}
Fix a basis $(\mathbfit{e}_1,\dotsc, \mathbfit{e}_{d_{\rho}})$ for $V$.
We first note that $\{\sigma_1^{-1},\dotsc, \sigma_m^{-1}\}$ is a full set of representatives of $G/H$, and the representation matrix of $\Ind_H^G\rho(\xi)$ under the basis 
\[
({\sigma_1}^{-1}\mathbfit{e}_1,\dotsc, {\sigma_1}^{-1}\mathbfit{e}_{d_{\rho}},{\sigma_2}^{-1}\mathbfit{e}_1,\dotsc, {\sigma_2}^{-1}\mathbfit{e}_{d_{\rho}},\dotsc, {\sigma_m}^{-1}\mathbfit{e}_1,\dotsc, {\sigma_m}^{-1}\mathbfit{e}_{d_{\rho}})
\]
for $\bigoplus_{i=1}^m\sigma_i K^{d_{\rho}}$ is given by
\begin{equation}
\begin{pmatrix}
\rho(\sigma_1\xi{\sigma_1}^{-1}) & \rho(\sigma_1\xi{\sigma_2}^{-1}) & \cdots & \rho(\sigma_1\xi{\sigma_m}^{-1}) \\
\rho(\sigma_2\xi{\sigma_1}^{-1}) & \rho(\sigma_2\xi{\sigma_2}^{-1}) & \cdots & \rho(\sigma_2\xi{\sigma_m}^{-1}) \\
\vdots&\vdots&\cdots&\vdots \\
\rho(\sigma_m\xi{\sigma_1}^{-1}) & \rho(\sigma_m\xi{\sigma_2}^{-1}) & \cdots & \rho(\sigma_m\xi{\sigma_m}^{-1})
\end{pmatrix},
\end{equation}
where $\rho(\xi)$ is defined by
\begin{equation}
\rho(\xi)=
\begin{cases}
\begin{minipage}{5.7cm}
the representation matrix of $\rho(\xi)$ under the standard basis
\vspace{5pt}
\end{minipage}& \text{if $\xi \in H$,} \\
\text{the square zero matrix of order $d_{\rho}$} & \text{if $\xi \not\in H$.}
\end{cases}
\end{equation}
Now we calculate 
\begin{align*}
\sum_{\tau \in H}&\left(\mathcal{W}_{Z}^{\beta,\tau}\otimes \rho(\tau)\right)\biggl(\bigl((u,H\sigma_i),\mathbfit{e}_k\bigr), \bigl((v,H\sigma_j),\mathbfit{e}_l\bigr)\biggr) \\
&=\sum_{\substack{(e, H\sigma_h)\in \mathbb{E}(Z)\colon\\ o(e, H\sigma_h)=(u,H\sigma_i), \\ t(e, H\sigma_h)=(v, H\sigma_j),\\ \beta(e, H\sigma_h)=\tau}}\sqrt{w_{(u, H\sigma_{i\vphantom{j}})}}\sqrt{w_{(v, H\sigma_j)}}\rho(\tau)(\mathbfit{e}_k,\mathbfit{e}_l) \\
&=\sum_{\substack{e\in \mathbb{E}(X)\colon\\ o(e)=u,\ t(e)=v,\\\sigma_i\alpha(e){\sigma_j}^{-1}\in H}}\sqrt{w_u}\sqrt{w_v}\rho(\sigma_i\alpha(e){\sigma_j}^{-1})(\mathbfit{e}_k,\mathbfit{e}_l) \\
&=\sum_{\xi \in G}\left(\sum_{\substack{e\in \mathbb{E}(X)\colon \alpha(e)=\xi\\ o(e)=u,\ t(e)=v}}\sqrt{w_u}\sqrt{w_v}\rho(\sigma_i\xi{\sigma_j}^{-1})(\mathbfit{e}_k,\mathbfit{e}_l) \right) \\
&=\sum_{\xi \in G}\left(\mathcal{W}_X^{\alpha, \xi}(u,v)\cdot \left(\Ind_H^G\rho(\xi)\right)\bigl(({\sigma_i}^{-1}\mathbfit{e}_k, {\sigma_j}^{-1}\mathbfit{e}_l)\bigr)\right) \\
&=\sum_{\xi \in G}\left(\biggl(\mathcal{W}_X^{\alpha, \xi}\otimes \left(\Ind_H^G\rho(\xi)\right)\biggr)\bigl((u, {\sigma_i}^{-1}\mathbfit{e}_k), (v, {\sigma_j}^{-1}\mathbfit{e}_l)\bigr)\right),
\end{align*}
which establishes the formula.
\end{proof}
\begin{theorem}[Induction Property]
Let $X$ be a $K$-valued vertex-weighted finite graph, let $\alpha\colon \mathbb{E}(X) \to G$ be a voltage assignment with a finite group $G$, and let $H$ be a subgroup of $G$.
We write $Z$ for the intermediate graph corresponding to $H$ given by \eqref{eq_22}.
Then 
\begin{equation}\label{eq_23}
h_Z^{\beta}(\rho, t)=h_X^{\alpha}\left(\Ind_H^G\rho, t\right)
\end{equation}
for every representation $\rho$ of $H$.
\end{theorem}
\begin{proof}
\eqref{eq_23} allows us to calculate 
\begin{align*}
h&_Z^{\beta}(\rho, t)\\
&=\det\left(
\mathbfit{I}_{V(Z)}\otimes \mathbfit{I}_{d_{\rho}}-t\cdot\sum_{\tau \in H}\mathcal{W}_Z^{\beta, \tau}\otimes \rho(\tau) +t^2\cdot \left(\mathbfit{D}_Z-\mathbfit{I}_{V(Z)}\right)\otimes \mathbfit{I}_{d_{\rho}}\right)\\
&=\det\left(
\mathbfit{I}_{V(X)}\otimes\mathbfit{I}_{H\backslash G}\otimes \mathbfit{I}_{d_{\rho}}-t\cdot\sum_{\xi \in G}\left(\mathcal{W}_X^{\alpha, \xi}\otimes \left(\Ind_H^G\rho(\xi)\right)\right) +t^2\cdot \left(\mathbfit{D}_X-\mathbfit{I}_{V(X)}\right)\otimes\mathbfit{I}_{H\backslash G}\otimes \mathbfit{I}_{d_{\rho}}\right) \\
&=\det\left(
\mathbfit{I}_{V(X)}\otimes \mathbfit{I}_{d_{\Ind_H^G\rho}}-t\cdot\sum_{\xi \in G}\left(\mathcal{W}_X^{\alpha, \xi}\otimes \left(\Ind_H^G\rho(\xi)\right)\right) +t^2\cdot \left(\mathbfit{D}_X-\mathbfit{I}_{V(X)}\right)\otimes \mathbfit{I}_{d_{\Ind_H^G\rho}}\right) \\
&=h_X^{\alpha}\left(\Ind_H^G\rho, t\right).
\end{align*}
This completes the proof.
\end{proof}
\section{Vertex-weighted Iwasawa's class number formula}\label{sec_4}
\subsection{Asymptotic theory of Cuoco and Monsky}\label{subsec_4.1}
In this subsection, we recall the Iwasawa $\mu$-invariant and the Iwasawa $\lambda$-invariant for a formal power series over a finite extension $K$ of $\mathbb{Q}_p$ following \cite{MR614400}, \cite{monsky}, and \cite[Section 3.2]{AMT}.

Fix an algebraic closure $\overbar{\mathbb{Q}}_p$ of $\mathbb{Q}_p$, and fix an embedding $\mathbb{Q}_p\hookrightarrow \overbar{\mathbb{Q}}_p$.
Let $K$ be a finite extension of $\mathbb{Q}_p$ with the valuation $\val_p$ normalized so that $\val_p(p)=1$, let $\mathcal{O}$ be the valuation ring of $K$, and let \glssymbol{Lambdad}$\Lambda_{d,\mathcal{O}}$ be the ring $\mathcal{O}\lBrack T_1,\dotsc, T_d \rBrack$ of formal power series.
Let \glssymbol{Gamma}$\Gamma$ be a multiplicative group isomorphic to $\mathbb{Z}_p^d$.
Recall that if we fix a basis $(\sigma_1,\dotsc, \sigma_d)$ for $\Gamma$, then we have an Iwasawa-Serre isomorphism $\mathcal{O}\lBrack \Gamma\rBrack \similarrightarrow \Lambda_{d,\mathcal{O}}$ that sends $\sigma_i$ to $1+T_{i}$.
Hence $\Gamma$ can be seen as a subgroup of $\Lambda_{d,\mathcal{O}}$ whose elements are of the form $(1+T_1)^{r_1}\cdots(1+T_d)^{r_d}$, $r_1,\dotsc, r_d \in \mathbb{Z}_p$.
Note that for $\sigma = \sigma_1^{a_1}\cdots\sigma_d^{a_d} \in \Gamma \setminus \Gamma^p$, $p \not\mid a_i$ for some $i$, and so $(\sigma_1,\dotsc, \sigma_{i-1},\sigma, \sigma_{i+1},\dotsc, \sigma_d)$ forms a basis for $\Gamma$.
Hence we obtain another Iwasawa-Serre isomorphism that sends $\sigma$ to $1+T_1$.
Let $\mathfrak{m}$ be the maximal ideal of $\mathcal{O}$, and let $\mathbb{F}_q$ be the finite field of order $q$ such that $\mathbb{F}_q\cong \mathcal{O}/\mathfrak{m}$.
We write \glssymbol{Omegad}$\Omega_{d,\mathcal{O}}\cong \mathbb{F}_q\lBrack T_1,\dotsc,T_d\rBrack$ for $\Lambda_{d,\mathcal{O}}/(\mathfrak{m}\Lambda_{d,\mathcal{O}}) $, and we write $(-)\overbar{\hphantom{I}}$ for the reduction map $\Lambda_{d,\mathcal{O}} \to \Omega_{d,\mathcal{O}}$.
Then the fact $\overbar{T}_1$ is a prime element of the unique factorization domain $\Omega_{d,\mathcal{O}}$ tells us that $(\sigma-1)\overbar{\hphantom{I}}$ is also a prime element of $\Omega_{d,\mathcal{O}}$.
Therefore, we can consider the order function $\ord_{\mathfrak{p}}\colon \Omega_{d,\mathcal{O}} \to  \mathbb{Z}_{\geq 0} \cup \{\infty\}$ associated to the prime ideal of the form $\mathfrak{p}=\left\langle (\sigma-1)\overbar{\hphantom{)}}\right\rangle$ for some element $\sigma \in \Gamma\setminus \Gamma^{p}$.
\begin{definition}
Let $\varpi$ be a uniformizer of the discrete valuation ring $\mathcal{O}$, and let $F$ be a non-zero element of $\Lambda_{d,\mathcal{O}} \otimes_{\mathcal{O}}K$.
Since $F$ is a finite sum of elements of the form $G\otimes \varpi^i$, where $G$ is an element of $\Lambda_{d,\mathcal{O}}$ and $i$ is an integer,
there exist a unique integer $N(F)$ and an unique element $F_0 \in \Lambda_{d,\mathcal{O}}$ such that $\varpi \not\mid F_0$ and $F=\varpi^{N(F)}F_0$.
Now we define the generalized Iwasawa invariants.
\begin{enumerate}[label=(\roman*)]
\item
The \emph{generalized Iwsawa \glssymbol{muF}$\mu$-invariant $\mu(F)$ of $F$} is the rational number defined by 
\begin{equation}
\mu(F)=\frac{N(F)}{e(K/\mathbb{Q}_p)},
\end{equation}
where $e(K/\mathbb{Q}_p)$ stands for the ramification index of $K/\mathbb{Q}_p$.
\item
The \emph{generalized Iwsawa \glssymbol{lambdaF}$\lambda$-invariant $\lambda(F)$ of $F$} is the non-negative integer defined by 
\begin{equation}
\lambda(F)=\sum\ord_{\mathfrak{p}}(F_0),
\end{equation}
where the sum extends over all prime ideals $\mathfrak{p}$ of the form $\mathfrak{p}=\left\langle (\sigma-1)\overbar{\hphantom{)}}\right\rangle$ for some element $\sigma \in \Gamma\setminus \Gamma^{p}$.
\end{enumerate}
\end{definition}
Now we state the asymptotic formula.
Let $W=\left\{\,\zeta \in \overbar{\mathbb{Q}}_p\;\middle|\; \text{$\zeta^{p^m}=1$ for some $m \in \mathbb{
Z}_{\geq 0}$}\,\right\}$, and let $W^d(n)=\left\{\,\mathbfit{\zeta}\in W^d\;\middle|\; \mathbfit{\zeta}^{p^n}=\mathbfit{1}\,\right\}$.
\begin{theorem}[{\cite[Theorem 3.5]{AMT}}]\label{thm_1}
Let $F$ be a non-zero element of $\Lambda_{d,\mathcal{O}}\otimes_{\mathcal{O}}K$.
Suppose that the function $\mathbfit{\zeta}\mapsto F(\mathbfit{\zeta}-\mathbfit{1})$ does not vanish on $W^d$.
Then there exist numbers $\mu_1,\dotsc,\mu_{d-1},$ $\lambda_1,\dotsc, \lambda_{d-1},$ $\nu \in\mathbb{Q}$ such that 
\begin{equation}\label{eq_32}
\sum_{\mathbfit{\zeta} \in W^d(n)\setminus \{\mathbfit{1}\}}\val_p\left(F\left(\mathbfit{\zeta}-\mathbfit{1}\right)\right)
=\left(\mu(F)\cdot p^n+\lambda (F) \cdot n\right)p^{(d-1)n}
+\sum_{i=1}^{d-1}\left(\mu_i\cdot p^n+\lambda_i\cdot n\right)p^{(d-1-i)n}
+\nu
\end{equation}
for all $n \gg 0$.
\end{theorem}
\begin{proof}
See \cite[Theorem 3.5]{AMT}.
\end{proof}
\begin{remark}\label{rem_1}
When $d=1$, we can find a number $n_0$ such that \eqref{eq_32} holds for all $n\geq n_0$.
Choose a number $n_0$ so that $\varphi\bigl(p^{n_0}\bigr) > \lambda(F)\cdot e(K/\mathbb{Q}_p)$, where $\varphi$ is Euler's totient function, and then $n_0$ is a desired number.
See the proof of \cite[Theorem 2.1]{monsky} for more details.
\end{remark}
\subsection{A proof of Iwasawa-type formula}\label{subsec_4.2}
Fix an embedding $\mathbb{Q}_p\hookrightarrow \overbar{{\mathbb{Q}}}_p$, and fix a finite extension $K$ of $\mathbb{Q}_p$.
From now on, a \emph{vertex-weighted graph} means a $K$-valued weighted graph.
Let $X$ be a vertex-weighted finite graph, and choose a square root $\sqrt{w_v}$ for each $v \in V(X)$.
Let \(\zeta_p\) be a primitive \(p\)-th root of unity.
Consider $K'=K\left(\left\{\,\sqrt{w_v}\;|\;v\in V(X)\,\right\}\cup\{\zeta_p\}\right)$ and its valuation ring \(\mathcal{O}'\).
The reason why we add \(\zeta_p\) into \(K'\) is that, in Definitions 5.2 below, we will treat a series that has \(\zeta_p\) in its coefficients.
In this subsection, we prove an Iwasawa-type formula for vertex  weighted graphs.
Set \glssymbol{Gamman}$\Gamma_n=\Gamma/\Gamma^{p^n} \cong (\mathbb{Z}/p^n\mathbb{Z})^d$.
Let $\alpha\colon \mathbb{E}(X)\to \Gamma$ be a voltage assignment.
Then, using the natural map $\Gamma \to \Gamma_n$, we obtain voltage assignments $\alpha_n\colon \mathbb{E}(X) \to \Gamma_n$ for each $n \in \mathbb{Z}_{\geq 0}$, and hence the tower of covers 
\begin{equation}
X=X(\alpha_0) \leftarrow X(\alpha_1) \leftarrow X(\alpha_2)\leftarrow \cdots \leftarrow X(\alpha_n) \leftarrow \cdots \leftarrow X(\alpha),
\end{equation}
which is called the \emph{$\mathbb{Z}_p^d$-tower over $X$}.
Note that by \eqref{eq_14}, $\kappa_{(v, \sigma)}(X(\alpha_n))=\kappa_{(v, \tau)}(X(\alpha_n))$ for every $\sigma,\ \tau \in \Gamma_n$.
Henceforth we would rather write $\kappa_v(X(\alpha_n))$ than $\kappa_{(v, \sigma)}(X(\alpha_n))$.
\begin{definitions}\ 
\begin{enumerate}[label=(\roman*)]
\item
For $\sigma=\sigma_1^{a_1}\cdots\sigma_d^{a_d} \in \Gamma$, set
\glssymbol{bbtbma}$\mathbb{t}(\sigma)(T_1,T_2,\dotsc, T_d)=(1+T_1)^{a_1}(1+T_2)^{a_2}\cdots (1+T_d)^{a_d} \in \Lambda_{d,\mathcal{O}'}$.
\item
\glssymbol{bbWXalpha}$\mathbb{W}_X^{\alpha}(T_1,\dotsc, T_d)$ over $\Lambda_{d,\mathcal{O}'}\otimes_{\mathcal{O}'}K'$ whose $(u,v)$-entry is given by
\begin{equation}
(\mathbb{W}_X^{\alpha}(T_1,\dotsc, T_d))(u,v)=\sum_{\substack{e \in \mathbb{E}(X)\colon \\ o(e)=u,\ t(e)=v}}\mathbb{t}(\alpha(e))(T_1,\dotsc, T_d)\sqrt{w_u}\sqrt{w_v}.
\end{equation}
\item
We define the element \glssymbol{QXalpha}$Q_X^{\alpha}(T) \in \Lambda_{d,\mathcal{O}'}\otimes_{\mathcal{O}'}K'$ by
\begin{equation}
Q_X^{\alpha}(T_1,\dotsc, T_d)=\det \left(\mathbfit{D}_X - \mathbb{W}_X^{\alpha}(T_1,\dotsc, T_d)\right).
\end{equation}
\item For an irreducible representation (namely, a character) $\psi \in \widehat{\Gamma}_n=\hom(\Gamma_n, {\overbar{\mathbb{Q}}_p}^{\ast})$, we use the symbol \glssymbol{bmzetapsi}$\mathbfit{\zeta}_{\psi}$ to denote 
\begin{equation}
\mathbfit{\zeta}_{\psi}=\left(\psi(\overbar{\sigma}_1),\psi(\overbar{\sigma}_2),\dotsc, \psi(\overbar{\sigma}_d)\right)\in (\overbar{\mathbb{Q}}_p^*)^d.
\end{equation}
\end{enumerate}
\end{definitions}
\begin{proposition}\label{prop_1}
For $\psi\in \widehat{\Gamma}_n$, 
\begin{equation}\label{eq_20}
Q_X^{\alpha}\left(\mathbfit{\zeta}_{\psi}-\mathbfit{1}\right)=h_X^{\alpha_n}(\psi, 1).
\end{equation}
\end{proposition}
\begin{proof}
In general, for every $\sigma=\sigma_1^{a_1}\cdots\sigma_d^{a_d} \in \Gamma$,
\begin{align*}
\mathbb{t}(\sigma)\left(\mathbfit{\zeta}_{\psi}-\mathbfit{1}\right)
&=(\psi(\overbar{\sigma}_1))^{a_1}\cdots(\psi(\overbar{\sigma}_d))^{a_d} \\
&=\psi\left({\overbar{\sigma}_1}^{a_1}\right)\cdots\psi\left({\overbar{\sigma}_d}^{a_d}\right) \\
&=\psi\left(\overbar{\sigma}\right).
\end{align*}
Hence $\mathbb{t}(\alpha(e))\left(\mathbfit{\zeta}_{\psi}-\mathbfit{1}\right)=\psi(\alpha_n(e))$, and for $(u,v)$-entry of $\mathbb{W}_X^{\alpha}(\mathbfit{\zeta}_{\psi}-\mathbfit{1})$,
\begin{align*}
\mathbb{W}_X^{\alpha}(\mathbfit{\zeta}_{\psi}-\mathbfit{1})(u,v)
&=\sum_{\substack{e \in \mathbb{E}(X)\colon \\ o(e)=u, \\ t(e)=v}}\psi(\alpha_n(e))\sqrt{w_u}\sqrt{w_v} \\
&=\sum_{\sigma \in \Gamma_n}\psi(\sigma)\sum_{\substack{e \in \mathbb{E}(X)\colon \alpha_n(e)=\sigma \\ o(e)=u,\ t(e)=v}}\sqrt{w_u}\sqrt{w_v} \\
&=\sum_{\sigma \in \Gamma_n}\psi(\sigma)\cdot\mathcal{W}_X^{\alpha_n}(u,v),
\end{align*}
which implies that $\mathbb{W}_X^{\alpha}\left(\mathbfit{\zeta}_{\psi}-\mathbfit{1}\right)=\sum_{\sigma \in \Gamma_n} \psi(\sigma)\otimes \mathcal{W}_X^{\alpha_n}$.
Therefore,
\begin{align*}
Q_X^{\alpha}\left(\mathbfit{\zeta}_{\psi}-\mathbfit{1}\right) 
& =\det\left(\mathbfit{D}_X - \sum_{\sigma \in \Gamma_n} \psi(\sigma)\otimes \mathcal{W}_X^{\alpha_n}\right)\\
&=\det\left(\mathbfit{H}_X^{\alpha_n}(\psi, 1)\right)\\
&=h_X^{\alpha_n}(\psi, 1),
\end{align*}
which establishes the formula.
\end{proof}
Now we give an Iwasawa-type formula for rooted weighted complexities, which is an analogue of \cite[Theorem 3.9]{AMT}.
As was the case with \cite[Theorem 3.9]{AMT}, this is a generalization of \cite[Theorem 6.2]{DV23} and \cite[Theorem 4.3]{KM24}.
\begin{theorem}[Iwasawa-type formula for rooted weighted complexities]\label{thm_2}
Let $X$ be a vertex-weighted finite graph, and let $\alpha\colon \mathbb{E}(X) \to \Gamma$ be a voltage assignment.
We assume that all the $X(\alpha_n)$ are connected.
Then we have the following.
\begin{enumerate}[label=$(\arabic*)$]
\item If there exists a non-negative integer $n_0$ such that $\kappa_v(X(\alpha_{n_0}))=0$, then $\kappa_v(X(\alpha_n))=0$ for all $n \geq n_0$.
\item 
If $\kappa_v(X(\alpha_n)) \neq 0$ for all $n$, then there exist numbers $\mu_1,\dotsc,\mu_{d-1},\ \lambda_1,\dotsc, \lambda_{d-1},\ \nu \in \mathbb{Q}$ such that
\begin{equation}
\val_p(\kappa_v(X(\alpha_n))) = \left(\mu_v(X, \alpha)p^n+\lambda_v(X, \alpha)  n\right)p^{(d-1)n}+\left(\sum_{i=1}^{d-1}\left(\mu_i  p^n+\lambda_i n\right)p^{(d-i-1)n}\right)+\nu
\end{equation}
for all $n \gg 0$, where \glssymbol{muvXalpha}$\mu_v(X, \alpha)=\mu\left(Q_X^{\alpha}\right)$ and where
\[
\glssymbol{lambdavXalpha}\lambda_v(X, \alpha) =
\begin{cases}
\lambda\left(Q_X^{\alpha}\right)-1 & \text{if $d=1$,} \\
\hfill \lambda\left(Q_X^{\alpha}\right) \hfill\hfill & \text{if $d \geq 2$}.
\end{cases}
\]
\end{enumerate}
The value $\mu_v(X,\alpha)$ is called the \emph{rooted Iwasawa $\mu$-invariant of $(X,\alpha)$ at $v$}, and the value $\lambda_v(X,\alpha)$ is called the \emph{rooted Iwasawa $\lambda$-invariant of $(X,\alpha)$ at $v$}.
\end{theorem}
\begin{proof}[Proof of $(1)$]
By \eqref{eq_14}, for each $n \geq n_0$, there exists a number $c_n \in \overbar{\mathbb{Q}}_p$ such that
$\kappa_v(X(\alpha_n))=c_n\cdot \kappa_v(X(\alpha_{n_0}))$.
Hence, when $\kappa_v(X(\alpha_{n_0}))=0$, we have $\kappa_v(X(\alpha_n))=0$, which is our assertion.
\end{proof}
\begin{proof}[Proof of $(2)$]
The assumption $\kappa_v(X(\alpha_n)) \neq 0$ for every $n$ implies that $h_X^{\alpha_n}(\psi, 1) \neq 0$ for every $n$ and for every $\psi \in \widehat{\Gamma}_n \setminus \bigl\{\mathrm{triv}_{\Gamma_n}\bigr\}$ by \eqref{eq_14}, which means that $Q_X^{\alpha}$ is non-vanishing on $W^d$ by $\eqref{eq_20}$.
Hence, by \eqref{eq_14} again, for $n \gg 0$,
\begin{align*}
\val&_p(\kappa_v(X(\alpha_n))) \\
&= \val_p(\kappa_v(X))-\val_p(\#\Gamma_n)+\raisebox{-.8ex}{\bigg(}\sum_{\psi \in \widehat{\Gamma}_n \setminus \{\mathrm{triv}_{\Gamma_n}\}}\val_p\left(h_X^{\alpha_n}(\psi, 1)\right)\raisebox{-.8ex}{\bigg)} \\
&= \val_p(\kappa_v(X))-nd+\raisebox{-.8ex}{\bigg(}\sum_{\mathbfit{\zeta}\in W^d(n)\setminus \{\mathbfit{1}\}}\val_p\left(Q_X^{\alpha}\left(\mathbfit{\zeta}-\mathbfit{1}\right)\right)\raisebox{-.8ex}{\bigg)} \quad\text{(Apply \eqref{eq_20}.)} \\
&=
\val_p(\kappa_v(X))-nd+
\left(
\begin{gathered}
\left(\mu(Q_X^{\alpha})p^n+\lambda (Q_X^{\alpha}) n\right)p^{(d-1)n}\\
+\sum_{i=1}^{d-1}\left(\mu_ip^n+\lambda_i n\right)p^{(d-1-i)n}
+\nu
\end{gathered} 
\right)
\quad\text{(Apply Theorem \ref{thm_1}.)} \\
&=
\begin{cases}
\left(\mu(Q_X^{\alpha})p^n+\lambda (Q_X^{\alpha}) n\right)p^{(d-1)n}
+
\left(
\begin{gathered}
\sum_{i=1}^{d-2}\left(\mu_ip^n+\lambda_i n\right)p^{(d-1-i)n} \\
+(\mu_{d-1}+(\lambda_{d-1}-d)n)
\end{gathered}
\right)
+\left(\nu+\val_p(\kappa_v(X))\right) & \text{if $d \geq 2$,}\\
\hfill
\mu(Q_X^{\alpha})p^n+\left(\lambda (Q_X^{\alpha})-1\right) n +\left(\nu+\val_p(\kappa_v(X))\right) \hfill\hfill&\text{if $d=1$,}
\end{cases}
\end{align*}
which establishes the formula.
\end{proof}
\begin{theorem}[Iwasawa-type formula for weighted complexities]\label{thm_3}
Let $X$ be a vertex-weighted finite graph, and let $\alpha\colon \mathbb{E}(X) \to \Gamma$ be a voltage assignment.
We assume that all the $X(\alpha_n)$ are connected.
Then we have the following.
\begin{enumerate}[label=$(\arabic*)$]
\item If there exists a non-negative integer $n_0$ such that $\kappa(X(\alpha_{n_0}))=0$, then $\kappa(X(\alpha_n))=0$ for all $n \geq n_0$.
\item 
If $\kappa(X(\alpha_n)) \neq 0$ for all $n$, then there exist numbers $\mu_1,\dotsc,\mu_{d-1},\ \lambda_1,\dotsc, \lambda_{d-1},\ \nu \in \mathbb{Q}$ such that
\begin{equation}\label{eq_33}
\val_p(\kappa(X(\alpha_n))) = \left(\mu(X,\alpha) p^n+\lambda(X, \alpha) n\right)p^{(d-1)n}+\left(\sum_{i=1}^{d-1}\left(\mu_i p^n+\lambda_i n\right)p^{(d-1-i)n}\right)+\nu
\end{equation}
for all $n \gg 0$, where \glssymbol{muXalpha}$\mu(X,\alpha)=\mu\left(Q_X^{\alpha}\right)$ and where \glssymbol{lambdaXalpha}$\lambda(X,\alpha)=\lambda\left(Q_X^{\alpha}\right)$.
\end{enumerate}
The value $\mu(X,\alpha)$ is called the \emph{Iwasawa $\mu$-invariant of $(X,\alpha)$}, and the value $\lambda$ is called the \emph{Iwasawa $\lambda(X, \alpha)$-invariant of $(X,\alpha)$}.
\end{theorem}
\begin{proof}[Proof of $(1)$]
Using \eqref{eq_21}, the proof is similar to the proof in Theorem \ref{thm_2} (1).
\end{proof}
\begin{proof}[Proof of $(2)$]
The assumption $\kappa(X(\alpha_n)) \neq 0$ for every $n$ implies that $h_X^{\alpha_n}(\psi, 1) \neq 0$ for every $n$ and for every $\psi \in \widehat{\Gamma}_n \setminus \{\mathrm{triv}_{\Gamma_n}\}$ by \eqref{eq_21},  which means that $Q_X^{\alpha}(T)$ is non-vanishing on $W^d$ by $\eqref{eq_20}$.
So, by \eqref{eq_21} again, for $n \gg 0$,
\begin{align*}
\val_p&(\kappa(X(\alpha_n))) \\
&= \val_p(\kappa(X))+\raisebox{-.8ex}{\bigg(}\sum_{\psi \in \widehat{\Gamma}_n \setminus \{\mathrm{triv}_{\widehat{\Gamma}_n}\}}\val_p\left(h_X^{\alpha_n}(\psi, 1)\right)\raisebox{-.8ex}{\bigg)} \\
&= \val_p(\kappa(X))+\raisebox{-.8ex}{\bigg(}\sum_{\mathbfit{\zeta}\in W^d(n)\setminus \{\mathbfit{1}\}}\val_p\left(Q_X^{\alpha}\left(\mathbfit{\zeta}-\mathbfit{1}\right)\right)\raisebox{-.8ex}{\bigg)}\quad\text{(Apply \eqref{eq_20}.)} \\
&=
\val_p(\kappa(X))+
\left(
\begin{gathered}
\left(\mu(Q_X^{\alpha})p^n+\lambda (Q_X^{\alpha}) n\right)p^{(d-1)n}\\
+\sum_{i=1}^{d-1}\left(\mu_ip^n+\lambda_i n\right)p^{(d-1-i)n}
+\nu
\end{gathered} 
\right)
\quad\text{(Apply Theorem \ref{thm_1}.)} \\
&=\left(\mu(Q_X^{\alpha})p^n+\lambda (Q_X^{\alpha}) n\right)p^{(d-1)n}
+
\left(
\sum_{i=1}^{d-1}\left(\mu_ip^n+\lambda_i n\right)p^{(d-1-i)n} 
\right)
+\left(\nu+\val_p(\kappa(X))\right),
\end{align*}
which establishes the formula.
\end{proof}
\begin{example}\label{ex_2}
We give an example for Theorem \ref{thm_3}.
Consider the vertex-weighted graph $X$ and the voltage assignment $\alpha\colon \mathbb{E}(X) \to \mathbb{Z}_2$ in Figure \ref{fig_101}.
\begin{figure}
\centering
\begin{minipage}{8em}
\centering
\begin{tikzpicture}[scale=1]
\coordinate[label=below:$v_1$,draw, circle, inner sep=2pt, fill=black](v1)at(0,0);
\coordinate[label=above left:$v_2$,draw, circle, inner sep=2pt, fill=black](v2)at({2*cos(pi/3 r)}, {2*sin(pi/3 r)});
\coordinate[label=below:$v_3$,draw, circle, inner sep=2pt, fill=black](v3)at(2,0);

\draw(v1)[->, very thick] to  node[above left]{$e_1$} (v2);
\draw(v3)[->, very thick] to  node[below]{$e_3$} (v1);
\draw(v2)[->, very thick]  to node[right]{$e_2$} (v3);
\draw(v1)[->, very thick]to[out=210, in=510, loop] node[left]{$e_4$} ();
\end{tikzpicture}
\end{minipage}
\hspace{3em}
\begin{minipage}{12em}
The weights of $X$ is given by
\begin{flushleft}
$
\begin{cases}
w_{v_1}=\sqrt{2},\\
w_{v_2}=1,\\
w_{v_3}=1.
\end{cases}
$
\end{flushleft}
\phantom{$($}
\end{minipage}
\begin{minipage}{0.03\columnwidth}
\phantom{a}
\end{minipage}
\begin{minipage}{11em}
$\alpha \colon \mathbb{E}(X)\to \mathbb{Z}_2$ is given by
\begin{flushleft}
$
\begin{cases}
\alpha(e_1)=1, \\\alpha(e_2)=0, \\
\alpha(e_3)=0, \\\alpha(e_4))=1.
\end{cases}
$
\end{flushleft}
\end{minipage}
\caption{The vertex-weighted graph $X$ and the voltage assignment $\alpha$}\label{fig_101}
\end{figure}
Since \(e_4\) is a loop and \(\alpha_n(e_4)=1\) generates \(\mathbb{Z}/2^n\mathbb{Z}\), all the $X(\alpha_n)$ are connected.
Since 
\begin{align*}
Q_X^{\alpha}(T)&=\det
\begin{pmatrix}
2+2\sqrt{2}-\sqrt{2}(1+T)-\sqrt{2}(1+T)^{-1} & -\sqrt[4]{2}(1+T) & -\sqrt[4]{2} \\
-\sqrt[4]{2}(1+T)^{-1} & 1+\sqrt{2} & -1 \\
-\sqrt[4]{2}&-1&1+\sqrt{2}
\end{pmatrix}\\
&=-\sqrt{2}(1+\sqrt{2})^2(1+T)^{-1}T^2,
\end{align*}
we have \(\mu(Q_X^{\alpha}(T))=1/2\), \(\lambda(Q_X^{\alpha})=2\), and hence \(\mu_{v_1}(X,\alpha)=1/2\), \(\lambda_{v_1}(X,\alpha)=1\).
According to Theorem \ref{thm_6}, letting \(\zeta_{2^n}\) be a primitive \(2^n\)-th root of unity, we can calculate the rooted complexity at \(v_1\) as
\begin{align}
\kappa_{v_1}(X(\alpha_n))&=\frac{\kappa_{v_1}(X)}{2^n}\prod_{i=1}^{2^n-1}3\sqrt{2}(2-\zeta_{2^n}^i-\zeta_{2^n}^{-i})\\
&=\frac{2+2\sqrt{2}}{2^n}(3\sqrt{2})^{2^n-1}\prod_{i=1}^{2^n-1}(1-\zeta_{2^n}^i)(1-\zeta_{2^n}^{-i})\\
&= \frac{1}{2^{n-1}}(1+\sqrt{2})(3\sqrt{2})^{2^n-1}\bigg(\prod_{i=1}^{2^n-1}\big(1-\zeta_{2^n}^i\big)\bigg)^2\\
&=(1+\sqrt{2})(3\sqrt{2})^{2^n-1}2^{n+1}.
\end{align}
Therefore, we have \(\val_2(\kappa_{v_1}(X(\alpha_n)))=(1/2)2^n+n+(1/2)\), which agrees with our Iwasawa-type formula for \(X\).
In the same manner, for the non-rooted version, we also deduce that \(\val_2(\kappa(X(\alpha_n)))=(1/2)2^n+2n+(1/2)\) while \(\mu(X,\alpha)=\mu(Q_X^{\alpha})=1/2\), \(\lambda(X,\alpha)=\lambda(Q_X^{\alpha})=2\).
Note that we have \(\nu = 1/2\), and hence the Iwasawa \(\nu\)-invariant of a \(\mathbb{Z}_p\)-tower may in fact be a rational number.
\end{example}
\section{Vertex-weighted Kida's formula}\label{sec_5}
In this section, we establish Kida's formula for vertex-weighted graphs.
For simplicity, we shall write \(\mu_v(X)\), \(\lambda_v(X)\), \(\mu(X)\), \(\lambda(X)\) for \(\mu_v(X,\alpha)\), \(\lambda_v(X,\alpha)\), \(\mu(X,\alpha)\), \(\lambda(X,\alpha)\), respectively.
\begin{theorem}[Kida's formula for vertex-weighted graphs]\label{thm_40}
Let $X$ and $Y$ be vertex-weighted graphs, let $\alpha\colon \mathbb{E}(X) \to \Gamma$ and $\alpha'\colon \mathbb{E}(Y) \to \Gamma$ be voltage assignments, and let $\pi\colon Y \to X$ be a Galois cover of $p$-power degree such that $\Gal(Y(\alpha'_n)/X)\cong \Gal(Y/X)\times \Gamma_n$.
Figure \ref{fig_6} illustrates our setting.
\begin{figure}[t]\label{fig_6}
\begin{tikzcd}[row sep=10mm]
Y \arrow[->, d, "\pi"']&Y(\alpha'_1) \arrow[->, l] \arrow[->, dl]&Y(\alpha'_2) \arrow[->, l] \arrow[->, dll, shorten > =2mm]& \cdots \arrow[->, l] &Y(\alpha'_n) \arrow[->, l] \arrow[->,dllll, shorten > = 6mm]& \cdots \arrow[->, l] & Y(\alpha') \arrow[->, l] \\
X& X(\alpha_1) \arrow[->, l]&X(\alpha_2) \arrow[->, l] & \cdots \arrow[->, l] & X(\alpha_n) \arrow[->, l]& \cdots \arrow[->, l] &X(\alpha). \arrow[->, l]
\end{tikzcd}
\caption{Two $\mathbb{Z}_p$-towers. The $n$-th vertical arrow is $\Gal(Y/X)\times \Gamma_n$-cover.}
\end{figure}
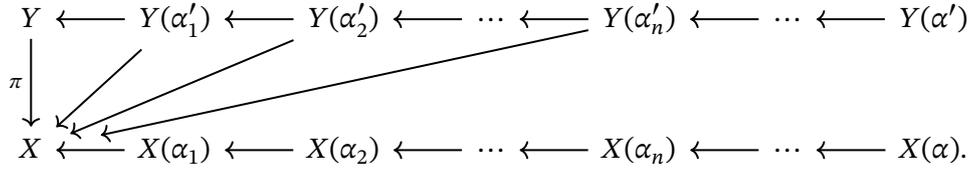
\noindent
If $(a)$, either $(b)$ or $(b)'$, and either $(c)$ or $(c)'$ 
\begin{enumerate}[leftmargin=39pt]
\item[$(a)\phantom{'}$] All the $Y(\alpha'_n)$ are connected.
\item[$(b)\phantom{'}$] All the $\kappa_v(Y(\alpha'_n))$ are non-zero for some fixed $v \in V(X)$. 
\item[$(b)'$] All the $\kappa(Y(\alpha'_n))$ are non-zero.
\item[$(c)\phantom{'}$] $\val_p(w_v) \geq \mu(X)/\#V(X) $ for every $v \in V(X)$.
\item[$(c)'$] $\val_p(w_v) \geq \mu(Y)/\#V(Y) $ for every $v \in V(Y)$.
\end{enumerate}
are satisfied, then the following $(1)$ and $(2)$ hold.
\begin{enumerate}[label=$(\arabic*)$]
\item 
$\mu\left(Q_Y^{\alpha'}\right)=[Y\colon X]\mu\left(Q_X^{\alpha}\right)$. Therefore, 
$\mu_v(Y) = [Y:X] \mu_v(X)$ for each $v \in V(X)$, and $\mu(Y)=[Y:X]\mu(X)$.
\item
$\lambda\left(Q_Y^{\alpha'}\right)=[Y\colon X]\lambda\left(Q_X^{\alpha}\right)$.
Therefore, 
\[
\lambda_v(Y)=
\begin{cases}
\hfill [Y:X]\lambda_v(X)\hfill\hfill& \text{if $d \geq 2$,} \\
[Y:X](\lambda_v(X)+1)-1 & \text{if $d=1$}.
\end{cases}
\]
for each $v \in V(X)$, and $\lambda(Y)=[Y:X]\lambda(X)$.
\end{enumerate}
\end{theorem}
To prove this theorem, we need some preliminaries.
\begin{definitions}\label{def_1}
Let $X$ be a vertex-weighted graph, and let $\alpha\colon \mathbb{E}(X) \to \Gamma$ and $\beta \colon \mathbb{E}(X) \to \mathbb{Z}/p\mathbb{Z}$ be voltage assignments.
\begin{enumerate}[label=(\roman*)]
\item
For each $\phi \in (\mathbb{Z}/p\mathbb{Z})\hat{\phantom{)}}$, define \glssymbol{bbWXalphabeta}$\mathbb{W}_X^{\alpha, \beta}(\phi, T_1,T_2,\dotsc, T_d)$ to be the matrix  over $\Lambda_{d,\mathcal{O}'}\otimes_{\mathcal{O}'} K'$ labeled by $V(X)$ and whose $(u,v)$-entry is given by
\begin{equation}
\left(\mathbb{W}_X^{\alpha, \beta}(\phi, T_1,T_2,\dotsc, T_d)\right)(u,v)
=\sum_{\substack{e\in \mathbb{E}(X)\colon \\ o(e)=u,\\ t(e)=v}}\phi(\beta(e))\cdot\mathbb{t}(\alpha(e))(T_1,\dotsc, T_d)\cdot\sqrt{w_u}\sqrt{w_v}.
\end{equation}
We often write simply $\mathbb{W}_X^{\alpha, \beta}(\phi,T)$ for $\mathbb{W}_X^{\alpha, \beta}(\phi, T_1,T_2,\dotsc, T_d)$.
\item
For each $\phi \in (\mathbb{Z}/p\mathbb{Z})\hat{\phantom{)}}$, define the series \glssymbol{QXalphabeta}$Q_X^{\alpha,\beta}(\phi,T_1,\dotsc, T_d)$ in $\Lambda_{d,\mathcal{O}'}\otimes_{\mathcal{O}'} K'$ to be 
\begin{equation}
Q_X^{\alpha,\beta}(\phi,T_1,\dotsc, T_d)
=\det\left(\mathbfit{D}_X- \mathbb{W}_X^{\alpha, \beta}(\phi, T_1,\dotsc, T_d)\right).
\end{equation}
We often write simply $Q_X^{\alpha, \beta}(\phi,T)$ for $Q_X^{\alpha, \beta}(\phi, T_1,T_2,\dotsc, T_d)$.
\item
For each non-negative integer $n$, define a map $\gamma_n\colon \mathbb{E}(X)\to (\mathbb{Z}/p\mathbb{Z})\times \Gamma_n$ to be given by
\begin{equation}
\gamma_n(e) = (\beta(e), \alpha_n(e))
\end{equation}
for each $e \in \mathbb{E}(X)$.
Note that $\gamma_n$ is a voltage assignment of $X$, and $X(\gamma_n) \cong (X(\beta))(\alpha_n \circ \pi)$, where $\pi$ stands for the Galois cover $X(\beta)/X$.
Figure \ref{fig_5} illustrates the above settings.
\begin{figure}
\begin{tikzcd}[row sep=3mm]
X(\beta) \arrow[<-, r] \arrow[->,dd, "\pi"']& (X(\beta))(\alpha_1\circ \pi)  \arrow[<-, r] &(X(\beta))(\alpha_2\circ \pi)  \arrow[<-, r] & \arrow[<-, r]  \cdots& (X(\beta))(\alpha\circ \pi) \arrow[->,dd]\\
{}&X(\gamma_1) \arrow[phantom, sloped, "\cong",u] \arrow[->, dl, shorten > =2mm] &X(\gamma_2) \arrow[phantom, sloped, "\cong",u] \arrow[->, dll, shorten > =6mm] &{}&{} \\
X \arrow[<-, r] & X(\alpha_1)  \arrow[<-, r] & X(\alpha_2)  \arrow[<-, r] & \arrow[<-, r] \cdots & X(\alpha)
\end{tikzcd}
\caption{Two $\mathbb{Z}_p$-towers.} \label{fig_5}
\end{figure}
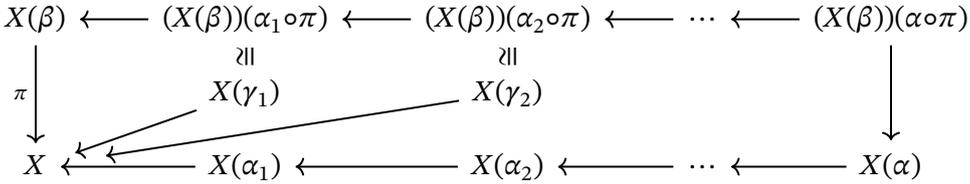
\end{enumerate}
\end{definitions}
\begin{lemma}
For every $\phi \in (\mathbb{Z}/p\mathbb{Z})\hat{\phantom{)}}$ and every $\psi \in \widehat{\Gamma}$, 
\begin{equation}\label{eq_24}
Q_X^{\alpha, \beta}\left(\phi, \mathbfit{\zeta}_{\psi}-\mathbfit{1}\right)=h_X^{\gamma_n}(\phi\boxtimes \psi, 1),
\end{equation}
where $\phi\boxtimes\psi$ is the external tensor product of $\phi$ and $\psi$.
\end{lemma}
\begin{proof}
This proof is similar to Proposition \ref{prop_1}, so we will give a sketch.
We first consider the matrix $\mathbb{W}_X^{\alpha, \beta}(\phi, T_1,\dotsc, T_d)$.
For the $(u,v)$-entry, we have
\begin{align*}
\left(\mathbb{W}_X^{\alpha, \beta}(\phi, \mathbfit{\zeta}_{\psi}-\mathbfit{1})\right)(u,v)
&=\sum_{\substack{e\in \mathbb{E}(X) \colon \\ o(e)=u,\\ t(e)=v}}\phi(\beta(e))\cdot \mathbb{t}(\alpha(e))\left(\mathbfit{\zeta}_{\psi}-\mathbfit{1}\right)\cdot \sqrt{w_u}\sqrt{w_v} \\
&=\sum_{\substack{e\in \mathbb{E}(X) \colon \\ o(e)=u,\\ t(e)=v}}\phi(\beta(e))\cdot \psi(\alpha_n(e))\cdot \sqrt{w_u}\sqrt{w_v} \\
&=\sum_{\substack{e\in \mathbb{E}(X) \colon \\ o(e)=u,\\ t(e)=v}}(\phi\boxtimes \psi)(\gamma_n(e))\cdot \sqrt{w_u}\sqrt{w_v} \\
&=\left(\sum_{\sigma \in  (\mathbb{Z}/p\mathbb{Z}) \times \Gamma_n}(\phi\boxtimes \psi)(\sigma)\otimes \mathcal{W}_X^{\gamma_n, \sigma}\right)(u,v),
\end{align*}
and so $\mathbb{W}_X^{\alpha, \beta}\left(\phi, \mathbfit{\zeta}_{\psi}-\mathbfit{1}\right)=\sum_{\sigma \in (\mathbb{Z}/p\mathbb{Z}) \times \Gamma_n}(\phi\boxtimes \psi)(\sigma)\otimes \mathcal{W}_X^{\gamma_n, \sigma}.$
Hence $Q_X^{\alpha, \beta}\left(\phi, \mathbfit{\zeta}_{\psi}-\mathbfit{1}\right)=h_X^{\gamma_n}(\phi\boxtimes \psi,1)$.
\end{proof}
\begin{lemma}
For every $\psi \in \widehat{\Gamma}_n$, 
\begin{equation}\label{eq_25}
\Ind_{\{0\}\times \Gamma_n}^{(\mathbb{Z}/p\mathbb{Z})\times \Gamma_n}(\mathrm{triv}_{(\mathbb{Z}/p\mathbb{Z})}\boxtimes \psi)\cong \bigoplus_{\phi \in (\mathbb{Z}/p\mathbb{Z})\hat{\phantom{)}}}\phi\boxtimes\psi.
\end{equation}
\end{lemma}
\begin{proof}
Recall that $((\mathbb{Z}/p\mathbb{Z})\times \Gamma_n)\hat{\phantom{)}}=\left\{\phi\boxtimes \xi \;\middle|\; \phi \in(\mathbb{Z}/p\mathbb{Z})\hat{\phantom{)}},\ \xi\in \widehat{\Gamma}_n \right\}$.
By the Frobenius reciprocity, we have
\begin{align*}
\biggl\langle \Ind_{\{0\}\times \Gamma_n}^{(\mathbb{Z}/p\mathbb{Z})\times \Gamma_n}(\mathrm{triv}_{(\mathbb{Z}/p\mathbb{Z})}\boxtimes \psi), \phi\boxtimes \xi\biggr\rangle &=\left\langle \mathrm{triv_{(\mathbb{Z}/p\mathbb{Z})}}\boxtimes \psi, \Res_{\{0\}\times \Gamma_n}^{(\mathbb{Z}/p\mathbb{Z})\times \Gamma_n}(\phi\boxtimes \xi)\right\rangle \\
&=\left\langle \mathrm{triv_{(\mathbb{Z}/p\mathbb{Z})}}\boxtimes \psi, \mathrm{triv}_{(\mathbb{Z}/p\mathbb{Z})}\boxtimes \xi\right\rangle \\
&=\langle \psi, \xi\rangle \\
&=\begin{cases}
1 & \text{if $\xi =\psi$}\\
0& \text{if $\xi \neq \psi$},\\
\end{cases}
\end{align*}
which means that $\Ind_{\{0\}\times \Gamma_n}^{(\mathbb{Z}/p\mathbb{Z})\times \Gamma_n}(\mathrm{triv}_{(\mathbb{Z}/p\mathbb{Z})}\boxtimes \psi)\cong \bigoplus_{\phi \in (\mathbb{Z}/p\mathbb{Z})\hat{\phantom{)}}}\phi\boxtimes\psi$, as desired.
\end{proof}
\begin{proposition}
For every $\psi \in \widehat{\Gamma}_n$,
\begin{equation}\label{eq_27}
Q_{X(\beta)}^{\alpha \circ \pi}\left(\mathbfit{\zeta}_{\psi}-\mathbfit{1}\right)=\prod_{\phi \in (\mathbb{Z}/p\mathbb{Z})\hat{\phantom{)}}}Q_X^{\alpha,\beta}\left(\phi, \mathbfit{\zeta}_{\psi}-\mathbfit{1}\right)
\end{equation}
\end{proposition}
\begin{proof}
By \eqref{eq_20}, \eqref{eq_23}, \eqref{eq_25}, \eqref{eq_26}, and \eqref{eq_24}, we have
\begin{align*}
Q_{X(\beta)}^{\alpha \circ \pi}\left(\mathbfit{\zeta}_{\psi}-\mathbfit{1}\right)
&=h_{X(\beta)}^{\alpha_n\circ \pi}(\psi, 1) &&\text{(Use \eqref{eq_20}.)}\\
&=h_{X}^{\gamma_n}\left(\Ind_{\{0\}\times \Gamma_n}^{(\mathbb{Z}/p\mathbb{Z})\times \Gamma_n}(\mathrm{triv}_{(\mathbb{Z}/p\mathbb{Z})}\boxtimes \psi), 1\right) &&\text{(Use \eqref{eq_23}.)}\\
&=h_X^{\gamma_n}\left(\bigoplus_{\phi \in (\mathbb{Z}/p\mathbb{Z})\hat{\phantom{)}}}\phi\boxtimes\psi, 1\right)&&\text{(Use \eqref{eq_25}.)} \\
&=\prod_{\phi\in (\mathbb{Z}/p\mathbb{Z})\hat{\phantom{)}}}h_X^{\gamma_n}(\phi\boxtimes \psi, 1)&&\text{(Use \eqref{eq_26}.)}\\
&=\prod_{\phi\in (\mathbb{Z}/p\mathbb{Z})\hat{\phantom{)}}}Q_X^{\alpha, \beta}\left(\phi, \mathbfit{\zeta}_{\psi}-\mathbfit{1}\right) &&\text{(Use \eqref{eq_24}.)},
\end{align*}
which completes the proof.
\end{proof}
\begin{corollary}
Suppose either 
\begin{enumerate}[leftmargin=39pt]
\item[$(b)\phantom{'}$] All the $\kappa_v(X(\gamma_n))$ are non-zero for some fixed $v \in V(X)$. 
\item[$(b)'$] All the $\kappa(X(\gamma_n))$ are non-zero. 
\end{enumerate}
are satisfied.
Then 
\begin{enumerate}[label=$(\arabic*)$]
\item
$Q_{X(\beta)}^{\alpha\circ \pi}$ is non-vanishing on $W^d$.
\item 
$Q_X^{\alpha,\beta}(\phi,T)$ is non-vanishing on $W^d$ for each $\phi \in (\mathbb{Z}/p\mathbb{Z})\hat{\phantom{)}}$.
\item
The Iwasawa $\mu$-invariant of $Q_{X(\beta)}^{\alpha\circ \pi}$ and the Iwasawa $\lambda$-invariant of $Q_{X(\beta)}^{\alpha\circ \pi}$ expand as follows.
\begin{align}\label{eq_28}
\mu\left(Q_{X(\beta)}^{\alpha\circ \pi}(T)\right)&=\sum_{\phi \in (\mathbb{Z}/p\mathbb{Z})\hat{\phantom{)}}}\mu\left(Q_{X}^{\alpha, \beta}(\phi, T)\right). \\
\label{eq_29}\lambda\left(Q_{X(\beta)}^{\alpha\circ \pi}(T)\right)&=\sum_{\phi \in (\mathbb{Z}/p\mathbb{Z})\hat{\phantom{)}}}\lambda\left(Q_{X}^{\alpha, \beta}(\phi,T)\right).
\end{align}
\end{enumerate}
\end{corollary}
\begin{proof}[Proof of $(1)$ and $(2)$]
The fact $Q_{X(\beta)}^{\alpha\circ\pi}(T)$ is non-vanishing on $W^d$ has been proven in Theorem \ref{thm_2} (if $(b)$ is satisfied) and in Theorem \ref{thm_3} (if $(b)'$ is satisfied).
Moreover, \eqref{eq_27} tells us that $Q_X^{\alpha,\beta}(\phi, T)$ is non-vanishing on $W^d$ for each $\phi \in (\mathbb{Z}/p\mathbb{Z})\hat{\phantom{)}}$, which is the desired conclusion.
\end{proof}
\begin{proof}[Proof of $(3)$]
Since (1) and (2) hold, we can apply Theorem \ref{thm_1} to $Q_{X(\beta)}^{\alpha\circ\pi}$ and $Q_X^{\alpha,\beta}(\phi,T)$ for each $\phi \in (\mathbb{Z}/p\mathbb{Z})\hat{\phantom{)}}$.
Now, using \eqref{eq_27}, the expansions directly follow.
\end{proof}
We are ready to prove Theorem \ref{thm_40}.
\begin{proof}[Proof of Theorem $\ref{thm_40}$ $(1)$]
We first claim that it suffices to show the formula in the case $G \cong \mathbb{Z}/p\mathbb{Z}$, following \cite[p.18]{RV22}.
Let $C$ denote the center of $G$. Then, since $G$ is of prime power order, $C$ is a non-trivial subgroup, and it is of $p$-power order by Lagrange's theorem.
By Sylow's theorem, there exists a subgroup $H$ of $C$ whose order is $p$.
By construction, $H$ is a normal subgroup of $G$.
Let $Z$ stand for an intermediate graph of $Y/X$ corresponding to $H$.
Then, the two covers $Y/Z$ and $Z/X$ are both Galois, the degree of $Y/Z$ is $p$-power, and the degree of $Z/X$ is exactly $p$.
Repeating the above discussion, we obtain the commutative diagram 
\[
\begin{tikzcd}[row sep=0.2cm]
X \arrow[phantom, d, "=", sloped]&&&& Y \arrow[->,llll] \arrow[phantom, d, "=", sloped]\\
Z_0 & Z_1\arrow[->,l] & Z_2 \arrow[->,l]& \cdots \arrow[->,l] & \arrow[->,l] Z_l
\end{tikzcd}
\]
of covers, where $Z_{i+1}/Z_i$ is a Galois cover of degree exactly $p$ for each $i \in \{0,1,\dotsc, l-1\}$.
As in Definitions \ref{def_1} (iii), we can construct $\mathbb{Z}_p^d$-towers
\[
Z_i \leftarrow Z_i\Big(\alpha^{(i)}_1\Big) \leftarrow Z_i\Big(\alpha^{(i)}_2\Big) \leftarrow \cdots \leftarrow Z_i\Big(\alpha^{(i)}\Big)
\]
satisfying $(a)$, either $(b)$ or $(b)'$, and either $(c)$ or $(c)'$.
This tells us that it is enough to show the theorem in the case $G \cong \mathbb{Z}/p\mathbb{Z}$.

Henceforth suppose $G=\mathbb{Z}/p\mathbb{Z}$.
We first give a proof for the case $(b)$ is satisfied.
Put $m=\#V(X)$, and put $M=\mu(X)e(K/\mathbb{Q}_p)$.  So $Q_X^{\alpha}(T)=\varpi^M Q_{X, 0}^{\alpha}(T)$, where $Q_{X, 0}^{\alpha}(T)$ is a polynomial in $\Lambda_{d,\mathcal{O}'}$ that can not be divided by $\varpi$.
Fix an $m$-th root $\varpi^{1/m}$ of $\varpi$, and write $K''$ for a field extension of $K'$  obtained by adding $\varpi^{1/m}$ to $K'$ equipped with the valuation $\val_p$ normalized so that $\val_p(p)=1$.
Let \(\mathcal{O}''\) be the valuation ring with maximal ideal \(\mathfrak{m}''\).
Write $\varpi^{-M/m}X$ for the $K''$-valued vertex-weighted graph whose weight function is given by $v \mapsto \varpi^{-M/m}w_v$.
 Because $\mu(X)/\#V(X) \leq \val_p(w_v)$ for every $v\in V(X)$, $\varpi^{-M/m}X$ is $\mathcal{O}''$-valued.
In these notations, $Q_{X,0}^{\alpha}(T)=\varpi^{-M}Q_X^{\alpha}(T)=Q_{\varpi^{-M/m}X}^{\alpha}(T)$, and hence $\mu(\varpi^{-M/m}X)=0$.
Since
 $\mathbb{W}_{\varpi^{-M/m}X}^{\alpha, \beta}(\phi, T)$ is a matrix over $\mathcal{O}''\lBrack T \rBrack$ for each $\phi \in (\mathbb{Z}/p\mathbb{Z})\hat{\phantom{)}}$, 
and since $\zeta_p-1 \in \mathfrak{m}''$, we have $\bigl(\mathbb{W}_{\varpi^{-M/m}X}^{\alpha,\beta}(\phi, T)\bigr)\overbar{\phantom{)}}=\bigl(\mathbb{W}_{\varpi^{-M/m}X}^{\alpha}(T)\bigr)\overbar{\phantom{)}}$ for every $\phi\in (\mathbb{Z}/p\mathbb{Z})\hat{\phantom{)}}$.
This implies that 
\begin{equation}\label{eq_30}
\bigl(Q_{\varpi^{-M/m}X}^{\alpha,\beta}(\phi, T)\bigr)\overbar{\phantom{)}} = \bigl(Q_{\varpi^{-M/m}X}^{\alpha}(T)\bigr)\overbar{\phantom{)}}, 
\end{equation}
and therefore $Q_{\varpi^{-M/m}X}^{\alpha,\beta}(\phi, T)$ can not be divided by $\varpi$, that is, $\mu\left(Q_{\varpi^{-M/m}X}^{\alpha,\beta}(\phi, T)\right)=0$ for every $\phi \in (\mathbb{Z}/p\mathbb{Z})\hat{\phantom{)}}$.
By \eqref{eq_28}, $\mu\left(Q_{\varpi^{-M/m}Y}^{\alpha\circ\pi}(T)\right)=0$.
Combining this conclusion with the fact that $Q_{\varpi^{-M/m}Y}^{\alpha \circ \pi}(T)=\varpi^{-p \cdot M}Q_Y^{\alpha \circ \pi}(T)$, we deduce that $\mu(Y)=p \cdot M=[Y:X]\mu(X)$, as desired.

For the case $(b)'$ is satisfied, put $M'=\mu(Y)e(K/\mathbb{Q}_p)$, and consider $\varpi^{-M'/(m\cdot p)}Y$ and $\varpi^{-M'/(m\cdot p)}X$, which are both $\mathcal{O}'$-valued.
Then $\mu\left(Q_{\varpi^{-M'/(m\cdot p)}X}^{\alpha, \beta}(\phi, T)\right)=0$ for each $\phi \in (\mathbb{Z}/p\mathbb{Z})\hat{\phantom{)}}$, and hence by \eqref{eq_28}, $\mu\left(Q_{\varpi^{-M'/(m\cdot p)}X}^{\alpha}(T)\right)=0$. 
This and the fact $Q_{\varpi^{-M'/(m\cdot p)}X}^{\alpha}(T)=\varpi^{-M'/p}Q_X^{\alpha}(T)$ imply $\mu(X)=p\cdot M'=[Y:X]\mu(Y)$.
This completes the proof.
\end{proof}
\begin{proof}[Proof of Theorem $\ref{thm_40}$ $(2)$]
We give a proof only for the case $(b)$ holds.
As in the above proof, it suffices to show the formula in the case $G \cong \mathbb{Z}/p\mathbb{Z}$. We use the same notations in the above proof.
By \eqref{eq_29} and \eqref{eq_30},
\begin{align*}
\lambda\left(Q_Y^{\alpha\circ\pi}(T)\right)
&=\lambda\left(Q_{\varpi^{-M/m}Y}^{\alpha\circ\pi}(T)\right)\\
&=\sum_{\phi \in (\mathbb{Z}/p\mathbb{Z})\,\hat{}}\lambda\left(Q_{\varpi^{-M/m}X}^{\alpha,\beta}(\phi, T)\right) \quad\text{(Use \eqref{eq_29}.)} \\
&=p\cdot \lambda\left(Q_{\varpi^{-M/m}X}^{\alpha}(T)\right)\quad \text{(Use \eqref{eq_30}.)} \\
&=p\cdot \lambda\left(Q_X^{\alpha}(T)\right).
\end{align*}
This implies that, $\lambda_v(Y)=[Y\colon X]\lambda_v(X)$, when $d \geq 2$, and $\lambda(Y)=[Y:X]\lambda(X)$.
When $d=1$, we have
\begin{align*}
\lambda_v(Y)+1&=\lambda\left(Q_Y^{\alpha\circ\pi}(T)\right) \\
&=[Y:X]\lambda\left(Q_X^{\alpha}(T)\right)\\
&=[Y:X](\lambda(X)+1),
\end{align*}
which establishes the formula.
\end{proof}
\begin{remark}
The statement of Kida's formula seems different from the one in \cite{RV22}.
When the non-weighted case (all the weights are 1), the condition either $(b)$ or $(b)'$ is satisfied is equivalent to the condition $\mu(X)=0$ or $\mu(Y)=0$, but then, $\mu(X)=\mu(Y)=0$ by (1).
This means that $\mu(X)=0$ if and only if $\mu(Y)=0$, and if this equivalent condition holds, then 
\[
\lambda(Y)=
\begin{cases}
\hfill [Y:X]\lambda(X) \hfill\hfill& \text{if $d \geq 2$,} \\
[Y:X](\lambda(X)+1)-1 & \text{if $d=1$,}
\end{cases}
\]
where $\lambda(X)$ and $\lambda(Y)$ stand for the Iwasawa $\lambda$-invariant for non-weighted graphs $X$ and $Y$ respectively.
Hence Theorem \ref{thm_40} is a direct generalization of Kida's formula for non-weighted graphs \cite[Theorem 4.1]{RV22}.
\end{remark}
\begin{example}\label{ex_3}
We give an example for Theorem \ref{thm_40}.
Consider the $\mathbb{Z}_2$-tower in Example \ref{ex_2}, and give another voltage assignment
$\beta\colon \mathbb{E}(X)\to \mathbb{Z}/2\mathbb{Z}$ defined by
\[
\beta(e_1)=0,\ \beta(e_2)=1,\ \beta(e_3)=0,\ \beta(e_4)=0.
\]
Put $Y = X(\beta)$.
Then, all the $X(\alpha_n)$ and all the $Y(\alpha_n\circ \pi)$ are connected, all the $\kappa(X(\alpha_n))$ and all the $\kappa(\alpha_n\circ \pi)$ are non-zero, and $\mu(X)=1/2, \lambda(X)=2$. Hence $(a)$, $(b)$, and $(c)$ are satisfied.
In order to calculate $\mu(Y)$ and $\lambda(Y)$, we calculate
\begin{align*}
&\phantom{a}Q_Y^{\alpha\circ \pi}(T)= \\
&\det\begin{pmatrix}
\begin{minipage}{3.1cm}$2+2\sqrt{2}-\sqrt{2}(1+T)-\sqrt{2}(1+T)^{-1}$\end{minipage}&0& -\sqrt[4]{2}(1+T) &0& -\sqrt[4]{2}&0 \\
0&\begin{minipage}{3.1cm}$2+2\sqrt{2}-\sqrt{2}(1+T)-\sqrt{2}(1+T)^{-1}$\end{minipage}&0& -\sqrt[4]{2}(1+T) & 0&-\sqrt[4]{2} \\
-\sqrt[4]{2}(1+T)^{-1} &0& 1+\sqrt{2} &0&0& -1 \\
0&-\sqrt[4]{2}(1+T)^{-1} &0& 1+\sqrt{2} & -1 &0\\
-\sqrt[4]{2}&0&0&-1&1+\sqrt{2}&0 \\
0&-\sqrt[4]{2}&-1&0&0&1+\sqrt{2}
\end{pmatrix}\\
\vphantom{\int}&\phantom{aQ_Y^{\alpha\circ \pi}(T)}=2\left(\left(-32-24\sqrt{2}\right)T+\left(76+64\sqrt{2}\right)T^2+\left(68+48\sqrt{2}\right)T^3+\left(11+4\sqrt{2}\right)T^4\right).\end{align*}
Hence $\mu(Y)=1,\ \lambda(Y)=4$, and Kida's formula holds.
Here we use \emph{SageMath} \cite{sagemath} to calculate $Q_X^{\alpha}(T)$ and $Q_Y^{\alpha\circ \pi}(T)$.
\end{example}
\begin{example}
There is a counterexample for the case $(c)$ in Theorem \ref{thm_40} is not satisfied.
Consider the same situation in Example \ref{ex_3}, but we change the weights of $X$ to the weights $w'$ given by
\[
w'_{v_1}=1,\ w'_{v_2}=1/2,\ w'_{v_3}=1.
\]
Then, all the $X(\alpha_n)$ and all the $Y(\alpha_n\circ \pi)$ are connected, and all the $\kappa(X(\alpha_n))$ and all the $\kappa(\alpha_n\circ \pi)$ are non-zero.
But
\begin{align*}
Q_X^{\alpha}(T)=\det
\begin{pmatrix}
\frac{7}{2}-(1+T)-(1+T)^{-1} & -\frac{1}{\sqrt{2}}(1+T) & -1 \\
-\frac{1}{\sqrt{2}}(1+T)^{-1} & 2 & -\frac{1}{\sqrt{2}} \\
-1&-\frac{1}{\sqrt{2}}&\frac{3}{2}\vphantom{\Bigg)}
\end{pmatrix}
=(-3(1+T)^{-1})T^2,
\end{align*}
which implies $\mu(X)=0, \lambda(X)=2$. This $\mathbb{Z}_2$-tower violates $(c)$ since $\val_2(w_{v_2})=-1$.
On the other hand 
\begin{align*}
&\phantom{a}Q_Y^{\alpha\circ \pi}(T)= \\
&\det\begin{pmatrix}
\frac{7}{2}-(1+T)-(1+T)^{-1} &0& -\frac{1}{\sqrt{2}}(1+T) &0& -1&0 \\
0&\frac{7}{2}-(1+T)-(1+T)^{-1}&0& -\frac{1}{\sqrt{2}}(1+T) & 0&-1 \\
-\frac{1}{\sqrt{2}}(1+T)^{-1} &0& 2 &0&0& -\frac{1}{\sqrt{2}} \\
0&-\frac{1}{\sqrt{2}}(1+T)^{-1}&0& 2 & -\frac{1}{\sqrt{2}} &0\\
-1&0&0&-\frac{1}{\sqrt{2}}&\frac{3}{2}\vphantom{\Biggl(}&0 \\
0&-1&-\frac{1}{\sqrt{2}}&0&0&\frac{3}{2}
\end{pmatrix}\\
\vphantom{\int}&\phantom{aQ_Y^{\alpha\circ \pi}(T)}=2\cdot(3(1+T)^{-2}(-1-T+T^2))T^2.
\end{align*}
Hence $\mu(Y)=1,\ \lambda(Y)=2$, and so Kida's formula fails.
Here we use \emph{SageMath} \cite{sagemath} to calculate $Q_Y^{\alpha\circ \pi}(T)$.
\end{example}
\clearpage
\printnoidxglossary[type=symbols, style=formel_altlong4colheader, title={Glossary of Symbols}]

\end{document}